\documentclass[11pt]{amsart}

\usepackage[normalem]{ulem}
\usepackage{url}
\usepackage{hyperref}
\usepackage{amssymb}
\usepackage{graphicx}
\usepackage{xspace}
\usepackage[usenames,dvipsnames]{color}
\usepackage{ifpdf}

\definecolor{brown}{cmyk}{0, 0.72, 1, 0.45}
\definecolor{grey}{gray}{0.5}

\newtheorem{theorem}{Theorem}
\newtheorem{lemma}[theorem]{Lemma}
\newtheorem{claim}[theorem]{Claim}

\newtheorem{remark}[theorem]{Remark}

\newtheorem{corollary}[theorem]{Corollary}
\newtheorem{definition}[theorem]{Definition}

\renewcommand{\l}{\lambda}
\renewcommand{\a}{\alpha}

\newcommand{\set}[1]{\{ #1 \}}
\newcommand{\card}[1]{\left| #1 \right|}

\newcommand{\bool}{\set{0,1}}

\newcommand{\bi}{\bigskip\noindent}
\newcommand{\si}{\smallskip\noindent}
\newcommand{\la}{\lambda}

\newcommand{\e}{\varepsilon}

\newcommand{\pr}{\mathbb{P}}
\renewcommand{\Pr}[1]{\mathbb{P} \! \left( {#1} \right)}
\newcommand{\ex}{\mathbb{E}}
\newcommand{\E}[1]{\mathbb{E}\left[ {#1} \right]}
\newcommand{\var}{\text{Var}}

\newcommand{\by}{\times}
\newcommand{\Po}{\operatorname{Po}}

\newcommand{\cstar}{c^*}
\newcommand{\cstark}{c_k^*}

\newcommand{\one}{\boldsymbol 1}
\newcommand{\zero}{\boldsymbol 0}
\newcommand{\Neqref}[1]{\textbf{[missing eqn ref]}}
\newcommand{\cond}{\mid}

\newcommand{\ab}{\bar{\a}}
\newcommand{\z}{\zeta}
\newcommand{\zn}{\z_0}

\newcommand{\ii}{\boldsymbol{i}}
\newcommand{\jj}{\boldsymbol{j}}
\newcommand{\uu}{\boldsymbol{u}}
\newcommand{\vv}{\boldsymbol{v}}
\newcommand{\zzz}{\boldsymbol{\zeta}}
\newcommand{\zzzargs}{\zzz(c,k,\a)}

\newcommand{\HH}{H_k(\a,\zzz;c)}
\newcommand{\D}{\Delta}
\renewcommand{\d}{\delta}
\newcommand{\ceil}[1]{\left\lceil {#1} \right\rceil}
\newcommand{\floor}[1]{\left\lfloor {#1} \right\rfloor}
\newcommand{\fa}{g}
\newcommand{\fb}{s}
\newcommand{\medup}{0.2743}

\newcommand{\twicelo}{0.4514}
\newcommand{\llo}{2.7694}
\newcommand{\cklo}{3.3992}
\newcommand{\rr}{R}

\newcommand{\leb}{\leq O(1) \,}
\def\parenssq(#1){\left( {#1} \right)}
\newcommand{\parens}[1]{\left( {#1} \right)}

\newcommand{\ak}{\a_k}
\newcommand{\dak}{\ak/3}
\newcommand{\aak}{0.99 \ak}

\newcommand{\Var}{\operatorname{Var}}
\newcommand{\trans}{\mathsf{T}} 
\newcommand{\AAmn}{\mathcal{A}_{m,n}}
\newcommand{\BBmn}{\mathcal{B}_{m,n}}
\newcommand{\CCmn}{\mathcal{C}_{m,n}}
\newcommand{\Xmn}{X_{m,n}}
\newcommand{\Xmnl}{\Xmn^{(\ell)}}
\newcommand{\Ymn}{Y_{m,n}}
\newcommand{\Ymnl}{\Ymn^{(\ell)}}
\newcommand{\DD}{\mathcal{D}}

\newcommand{\gbsxx}[1]{\relax}
\newcommand{\MM}{\overline{M}}
\newcommand{\textfrac}[2]{{#1} \left/ {#2} \right.}
\newcommand{\mm}{{M}}
\newcommand{\nn}{{N}}
\newcommand{\mnhyp}{$m,n \to \infty$ with $\lim m/n \in (2/k,1)$}
\newcommand{\mnhypi}{$m,n \to \infty$ with $\lim m/n \in (2/k,\infty)$}
\newcommand{\mnanyhyp}{$m,n \to \infty$ with $\liminf m/n>2/k$}
\newcommand{\Pola}{X(\la)}
\newcommand{\Posla}{X_s(\la)}
\newcommand{\gftwo}{\ensuremath{\mathbb{F}_2}\xspace}
\newcommand{\Ortn}{O\Big(\frac1{\sqrt n}\Big)}
\newcommand{\alnu}{a(\ell,\nu)}
\newcommand{\bmnu}{b(m-\ell,\nu)}
\newcommand{\omegan}{w(n)}
\newcommand{\im}{i}

\title[Satisfiability Threshold for $k$-XORSAT]
{The Satisfiability Threshold for \lowercase{k}-XORSAT}

\author[Boris Pittel]{Boris Pittel}
\thanks{$^\dagger$
This research was supported by DIMACS, Center for Discrete
Mathematics and Theoretical Computer Science, Rutgers, the State University
of New Jersey, funded by the NSF under Grant No.\
DMS06-02942, Special Focus on Discrete Random Systems, and
by the NSF under Grants No.\ DMS-0805996, No.\ DMS 1101237.
The paper was begun when the second author was a researcher in
the Department of Mathematical Sciences,
IBM T.J.\ Watson Research Center, Yorktown Heights NY 10598, USA%
}
\address[Boris Pittel]{
Department of Mathematics \\
Ohio State University \\
Columbus OH 43210, USA}
\email{bgp@math.ohio-state.edu}

\author[Gregory B. Sorkin]{Gregory B. Sorkin$^\dagger$}
\address[Gregory B. Sorkin]{
Department of Management \\
London School of Economics and Political Science \\
Houghton Street \\
London WC2A 2AE
}
\email{g.b.sorkin@lse.ac.uk}

\allowdisplaybreaks

\begin{document}

\bibliographystyle{amsalpha}

\date{\today}

\begin{abstract}
We consider ``unconstrained'' random $k$-XORSAT,
which is a uniformly random
system of $m$ linear non-ho\-mo\-ge\-ne\-ous 
equations in \gftwo
over $n$ variables,
each equation containing $k\ge 3$ variables,
and also consider a ``constrained'' model where
every variable appears in at least two equations.
Dubois and Mandler proved that
$m/n=1$ is a sharp threshold for satisfiability of
constrained $3$-XORSAT,
and analyzed the 2-core of a random 3-uniform hypergraph to
extend this result to find the threshold 
for unconstrained 3-XORSAT.

We show that $m/n=1$
remains a sharp threshold for satisfiability of 
constrained $k$-XORSAT
for every $k\ge 3$,
and we use standard results on the 2-core of a random $k$-uniform hypergraph
to extend this result to find the threshold 
for unconstrained $k$-XORSAT.
For constrained $k$-XORSAT we narrow the phase transition window,
showing that 
$n-m \to \infty$ implies almost-sure satisfiability, while
$m-n \to \infty$ implies almost-sure unsatisfiability.
\end{abstract}

\maketitle

\allowdisplaybreaks

\section{Introduction}
An instance of $k$-XORSAT is given by a set of $m$
linear equations in \gftwo, over $n$ variables,
each equation involving $k$ variables 
and a right hand side which is either 0 or~1.
Equivalently, it is a linear system $Ax=b$ modulo 2
in which $A$ is an $m \by n$ 0--1 matrix
each of whose row sums is $k$,
and $b$ is an arbitrary 0--1 vector.

Random instances of many problems of this sort undergo phase transitions
around some critical ratio $\cstar$ of $m/n$,
meaning that for $m,n \to \infty$ with $\lim m/n < \cstar$,
the probability that a random instance $F_{n,m}$ is satisfiable
(or possesses some similar property) approaches $1$,
while if $\lim m/n > \cstar$
the probability approaches $0$.
(There is no loss of generality in hypothesizing the existence of a limit
since, in a broad context, a result as stated 
implies the same with the weaker hypotheses 
$\liminf m/n>\cstar$ and $\limsup m/n<\cstar$.)
Friedgut~\cite{Friedgut} proved that a wide range of problems
have such sharp thresholds,
but with the possibility that the threshold $\cstar = \cstar(n)$
does not tend to a constant.
The relatively few cases in which $\cstar$ is known to be a constant include
2-SAT, by
Chv\'atal and Reed~\cite{ChRe}, Goerdt~\cite{Goerdt},
and Fernandez de la Vega~\cite{FdlV}
(with the scaling window detailed by
Bollob\'as, Borgs, Chayes, Kim, and Wilson, \cite{BBCKW}),
an extension to Max 2-SAT, by
Coppersmith, Gamarnik, Hajiaghayi, and Sorkin~\cite{CGHS},
and the pure-literal threshold for a $k$-SAT formula, by Molloy~\cite{Molloy}.

The most natural random model of the $k$-XORSAT problem is
the ``unconstrained'' model
in which
each of the $m$ equations'
$k$ variables are drawn uniformly (without replacement)
from the set of all $n$ variables,
and the right hand side values are uniformly 0 or 1;
equivalently a random instance $Ax=b$ is given by
a matrix $A \in \bool^{m \by n}$ drawn uniformly at random
from the set of all such matrices with each row sum equal to $k$,
and $b \in \bool^m$ chosen uniformly at random.

The case $k=2$ has been extensively studied. 
As shown by Kolchin~\cite{Kolchin}
and Creignon and Daud\'e~\cite{CrDa}, 
the random instance has a solution with limiting probability $p(2m/n)+o(1)$, 
where $p(x)\in (0,1)$ for $x<1$, $p(1-)=0$, and $p(x)\equiv 0$ for $x>1$. 
Daud\'e and Ravelomanana~\cite{DaRa}, and Pittel and Yeum~\cite{PiYe},
analyzed the near-critical behavior
of the solvability probability for $2m/n=1+\varepsilon$, 
$\varepsilon =o(n^{-1/4})$.

For $k>2$, Kolchin~\cite{Kolchin}
analyzed the expected number of nonempty ``critical row sets''
(nonempty collections of rows whose sum is all-even),
whose presence is necessary and sufficient for the (Boolean) rank of $A$ 
to be less than $m$. 
He determined the thresholds $c_k$ such that the expected number
of nonempty critical sets goes to 0 if $\lim m/n<c_k$ 
and to infinity if $\lim m/n>c_k$;
in particular, $c_3=0.8894\dots$.
Thus, for $\lim m/n < c_k$, with high probability
$A$ is of full rank, so $Ax=b$ is solvable.
It follows
that the satisfiability threshold $\cstar_k$ is at least $c_k$.
It is an easy observation (see Remark~\ref{>1})
that $\cstar_k \leq 1$.
However, Kolchin could not resolve the precise value, 
or even the existence, of the satisfiability threshold.

Dubois and Mandler~\cite{DMfocs} (see also \cite{DM02})
introduced a ``constrained'' random $k$-XORSAT model,
where $b$ is still uniformly random, but
$A$ is uniformly random over the subset of matrices in which
each column sum is at least~2.
For $k=3$ (3-XORSAT) they showed that its threshold for $m/n$ is~1.
This is of interest because from the threshold for the constrained model,
they were able to derive that for the unconstrained model.
Dubois and Mandler 
suggested that their methods could be extended 
to the general constrained $k$-XORSAT, $k\ge 3$. 
However, their approach 
--- the second-moment method for the number of solutions --- 
requires solving a hard maximization problem with $\Theta(k)$ variables, 
a genuinely daunting task.

Our main result is that 1 continues to be the threshold for all $k>3$.
\begin{theorem} \label{main}
Let $Ax=b$ be a uniformly random constrained
$k$-XORSAT instance with $m$ equations and $n$ variables.
Suppose $k\ge 4$. 
If \mnhyp\ 
then $Ax=b$ is almost surely satisfiable,
with satisfiability probability $1-O(m^{-(k-2)})$,
while if 
{$m,n \to \infty$ with $\lim m/n > 1$}
then $Ax=b$ is almost surely unsatisfiable,
with satisfiability probability $O(2^{-(m-n)})$.
\end{theorem}

We are also able to treat the case
when the gap between $m$ and $n$ is not linear
but arbitrarily slowly growing,
obtaining the following stronger theorem.
\begin{theorem} \label{sharp}
Let $Ax=b$ be a uniformly random constrained
$k$-XORSAT instance with $m$ equations and $n$ variables,
with $k \geq 3$ and
\mnanyhyp. 
Then, for any $\omegan \to +\infty$,
if $m \leq n-\omegan$
then $Ax=b$ is almost surely satisfiable,
with satisfiability probability $1-O(m^{-(k-2)}+\exp(-0.59 \: \omegan))$,
while if $m \geq n+\omegan$
then $Ax=b$ is almost surely unsatisfiable,
with satisfiability probability $O(2^{-\omegan})$.
\end{theorem}

Rather than using the second-moment method of Dubois and Mandler,
we use the critical-set approach of Kolchin.
Remark~\ref{equiv} shows that the two methods are equivalent,
but the second leads to more tractable calculations,
specifically, to a maximization problem with a number of variables
that is fixed, independent of $k$.
In this constrained model, 
by the same reasoning given above,
Kolchin's approach will show that $\cstar_k \geq 1$.
And, by the same reasoning as for the unconstrained model
(again see Remark~\ref{>1}),
we have $\cstar_k \leq 1$.
Thus, for the constrained model (unlike the constrained one), 
the two bounds coincide, establishing the threshold.

Dubois and Mandler extended the threshold for 
the constrained 3-XORSAT model to
that for the unconstrained model
by observing that, in an unconstrained instance,
any variable appearing in just one clause (or none), 
can be deleted along with that clause (if any),
to give an equivalent instance, and this process can be repeated.
The key observation is that 
a uniformly random \emph{unconstrained} instance
reduces to
a uniformly random \emph{constrained} instance
with a predictable edge density;
the threshold for the unconstrained model is the value for which
the corresponding constrained instance has density 1. 
The same approach works for any $k$, 
and we capitalize on existing analyses of the 2-core of a 
random $k$-uniform hypergraph to 
establish the unconstrained $k$-XORSAT threshold
in Theorem~\ref{unconstrained}.

\subsection*{Other related work}
Work on the rank of random matrices over finite fields is not as extensive 
as that on real random matrices, but nonetheless a survey is beyond our scope.
In addition to the work already described, we note that
the rank of matrices with independent random 0--1 entries 
was explored over a decade ago by
Bl\"omer, Karp and Welzl \cite{BlKaWe},
and Cooper \cite{Cooper2}, among others.

Concurrently with and independently from our work, 
the k-XORSAT phase transition was also analyzed by
Dietzfelbinger, Goerdt, Mitzenmacher, Montanari, Pagh, and Rink
as part of a study of cuckoo hashing~\cite{DGMMPR10,DGMMPR09}.
Recently, Darling, Penrose, Wade and Zabell have explored a random XORSAT
model replacing the constant $k$ with a distribution,
but the satisfiability threshold has not yet been determined for
this generalization.

To translate our result for the constrained model to the unconstrained one,
we exploit results on the core of a random hypergraph.
For usual graphs, the threshold for the appearance of an $r$-core
was first obtained by 
Pittel, Spencer, and Wormald \cite{PiSpWo}.
For $k$-uniform hypergraphs, the $r$-core thresholds were obtained
roughly concurrently by
Cooper \cite{Cooper1},
Kim \cite{Kim}, and
Molloy \cite{Molloy}.
Two aspects of Cooper's treatment are noteworthy.
First, he works with a degree-sequence hypergraph model;
taking Poisson-distributed degrees reproduces the results for a simple random
hypergraph.
Also, he observes
\cite[Section 5.2]{Cooper1} that the
point at which a random $k$-uniform hypergraph's core 
has a (typical) edges-to-vertices ratio of 1
is an upper bound on the satisfiability threshold of 
unconstrained $k$-XORSAT;
proving that this is the true threshold is the 
main subject of the present paper.

\subsection*{Outline}
The remainder of the paper is organized as follows.
Section~\ref{Prback}  formalizes our introductory observations about 
the first- and second-moment methods, 
the number of solutions, and the number of critical sets.
Section~\ref{PrSpaces} shows that
for the constrained model,
instead of considering random 0--1 matrices $A$,
it is asymptotically equivalent to consider random nonnegative 
\emph{integer} matrices $A$
subject to the same constraints on row sums (equal to $k$) 
and column sums (at least $2$).
Section~\ref{MainResult},
using generating functions and Chernoff's method, obtains an exponential bound
for the expected number of critical sets of any given cardinality.
Section~\ref{ProveLemma}
uses this bound
to show that, for $\lim m/n \in (2/k, 1)$ and $k>3$, 
the expected number of nonempty critical sets is $O(m^{-(k-2)})$.
Hence, with high probability,
there is no such set,
$A$ is of full rank, and the instance is satisfiable.
We conclude that 1 
is a sharp threshold for satisfiability of $Ax=b$ in
the constrained case for all $k\ge 3$.

Section~\ref{sharper} builds on the earlier results
to treat the case $\lim m/n=1$ and prove Theorem~\ref{sharp}.
Section~\ref{Unconstrained}
derives the unconstrained $k$-XORSAT threshold from the constrained one,
using standard results on the 2-core of a random hypergraph.

\section{Proof background}\label{Prback}

Let $N$ be the number of solutions 
to the system of equations $Ax=b$.

\begin{remark} \label{>1}
For an arbitrarily distributed $A \in \bool^{m \by n}$, 
with $b$ independent and uniformly distributed over $\bool^n$,
$\ex[N] = 2^{n-m}$,
and the satisfiability threshold is at most 1.
\end{remark}

\begin{proof}
Given $A$, there are $2^m$ systems given by $(A,b)$,
and in all they have $2^n$ solutions
since any $x$ uniquely determines $b=Ax$.
So $\E{N \cond A}=2^{n-m}$, and $\E{N}=2^{n-m}$.
By the first-moment method, 
$
\pr(Ax=b \text{ is satisfiable}) = \pr(N>0) \leq \ex[N] = 2^{n-m} ,
$
which tends to 0 if $\lim m/n >1$.
\end{proof}

\begin{definition}
Given a matrix, a \emph{critical set} is 
a collection of rows whose sum is all-even 
(i.e., the sum is the 0 vector in $\gftwo$).
\end{definition}
\noindent
Note that the
collection of critical sets is sandwiched between
the minimal linearly dependent sets of rows, and 
all linearly dependent sets of rows.

Let $X$ be the number of nonempty critical row subsets of a matrix $A$.
Where the first-moment method establishes the probable absence of solutions,
their probable presence can be established in this setting either by 
the second-moment method on the number of solutions,
showing that $\E{N^2} / \E{N}^2 \to 1$,
or by the first moment method on the number of non-empty critical row
sets, showing that $\E X \to 0$.
We will use the second approach (Kolchin's).
The two approaches suggest different calculations,
but as the following remark shows, they are equivalent.

\begin{remark} \label{equiv}
Let a distribution on $A \in \bool^{m \by n}$ be given, 
and let $b$ be independent of $A$ and uniformly distributed over $\bool^n$.
Then $\ex[N^2] \big/ \ex[N]^2 = \ex[X] +1$.
\end{remark}

\begin{proof}
Consider any fixed $A$, having rank $r(A)$ 
over $\gftwo$.
By elementary linear algebra, 
for each of the $2^{r(A)}$ values of $b$ in $\{Ax \colon x\in \{0,1\}^n\}$,
$Ax=b$ has $2^{n-r(A)}$ solutions,
giving $2^{2n-2\,r(A)}$ ordered pairs of solutions in each such case.
For the remaining values of $b$ there are no solutions,
so in all there are $2^{2n-r(A)}$ ordered pairs of solutions.
Taking the expectation over $b$ uniformly distributed over its $2^m$
possibilities, 
$\ex[N^2 \cond A] = \ex[2^{2n-r(A)-m}]$,
thus
$\ex[N^2] = \ex[2^{2n-r(A)-m}]$.
Since $\E{N}=2^{n-m}$ (see Remark~\ref{>1}),
\begin{equation*}
 \ex[N^2] / \ex[N]^2
  = \ex[2^{2n-r(A)-m}] / (2^{n-m})^2
  = \ex[2^{m-r(A)}] 
  = \ex[2^{n(A^\trans)}] ,
\end{equation*}
where $n(A^\trans)$ denotes the nullity of the transpose of $A$.

On the other hand, 
a critical row set is precisely one given by an indicator vector
$y \in \bool^m$ for which $y^\trans A=0$.
For a given $A$
the number of critical sets is thus $2^{n(A^\trans)}$,
and the expected number of non-empty critical row subsets is 
$\ex[X]=\E{2^{n(A^\trans)}}-1$.
\end{proof}

In fact, if $m \leq n$ and $\ex[X] \to 0$, then
with high probability $N=2^{n-m}$
(not merely $N/2^{n-m} \to 1$ in probability
as given by the second-moment method).
This follows because $X=0$ implies $r(A)=m$, in which case
$N=2^{n-m}$ for every $b$.
Thus,
$\Pr{N = 2^{n-m}}
 \geq \Pr{X=0}
 = 1- \Pr{X>0}
 \geq 1-\E X
 \to 1.
$

\si

The work in Sections \ref{PrSpaces}--\ref{ProveLemma} 
is to count the critical row subsets.
We will
show that indeed $\ex[X]\to 0$ for the constrained random model
with $k\ge4$ and \mnhyp.

\section{Probability spaces}  \label{PrSpaces}
This section will establish Corollary~\ref{preq},
showing that the uniform distribution over constrained $k$-XORSAT matrices
$A \in \AAmn$ (see below)
is for our purposes equivalent
to a model $C \in \CCmn$ 
allowing a variable to count as appearing more than once within an equation.

Let $\AAmn$ denote the set of all
$m\times n$ matrices with 0--1 entries, such that all $m$ row
sums are $k$, and all $n$ column sums are at least 2.
For $\AAmn$ to
be nonempty it is necessary that $km\ge 2n$,
and we will assume that 
\mnhyp.

A matrix $A\in \AAmn$ may be interpreted
as an outcome of the following allocation scheme. We have an
$m\times n$
array of cells with $k$ \emph{indistinguishable} chips assigned to each of
the $m$ rows. For each row, the $k$ chips are put in $k$ distinct cells
(so there is at most one chip per cell), subject
to the constraint that each column gets at least two chips.

Let us consider an alternative model, with the same constraints but
where the chips in each row are \emph{distinguishable},
giving allocations $B \in \BBmn$.
Then each allocation in $\AAmn$ is obtained
from $(k!)^m$ allocations in $\BBmn$,
and the uniform distribution on $\AAmn$ is equivalent to that on $\BBmn$.

Let $\CCmn$ be a relaxed version of $\BBmn$, without the requirement
that each of the $mn$ cells gets at most one chip.
Let $B$ and $C$ be
distributed uniformly on $\BBmn$ and $\CCmn$, respectively.
Crucially, and obviously, $B$ is equal in distribution to $C$, conditioned on $C\in
\BBmn$.

To state a key lemma on $\card\AAmn$, $\card\BBmn$, and
$\card\CCmn$ we need some notation, much of which will recur throughout the paper.
Introduce
\begin{equation} \label{fdef}
f(x)=\sum_{j\ge 2}\frac{x^j}{j!}=e^x-1-x
\qquad \text{ and } \qquad
\psi(x) := \frac{xf'(x)}{f(x)} .
\end{equation}
Define $\psi(0)=2$ by continuity. 

\begin{remark}\label{psi''}
$\psi(x)$ is strictly increasing. 
\end{remark}
\begin{proof}
From \eqref{fdef},
$
\psi'(x) =
\textfrac{(e^{2x} +1-2 e^x -x^2 e^x) }{ (e^x -x-1)^2} .
$
For $x \neq 0$ this is equal in sign to $e^{-x}$ times its numerator, i.e., to
$e^x+e^{-x}-2-x^2 = 2 (\cosh(x)-1-\tfrac12 x^2)$.
For $x \neq 0$ this is positive, 
as is immediate from the Taylor series expansion for $\cosh$.
Thus $\psi'(x)>0$ for $x \neq 0$,
and with continuity of $\psi'$ at 0 (easily checked) this proves the lemma.
\end{proof}

\gbsxx{
Any root of
$\psi(x) = km/n$
is a stationary point of $\,n\ln f(x) - km\ln x$
(its derivative is $(n/x) (\psi(x)-km/n)$).
}
Under our assumption that $m/n>2/k$,  
the equation
$\psi(x) = km/n$
has a unique root, and it is positive.
This follows from the facts that
$\psi(x)$ is strictly increasing (see Remark~\ref{psi''}), 
 $\psi(0+)=2$,
and $\psi(x) \to \infty$ as $x \to \infty$.
Henceforth, let 
\begin{align} \label{lambda}
\la &= \la(km/n)
 := \psi^{-1}(km/n)
\end{align}
be this root. 

Introduce a \emph{truncated} Poisson random variable $Z=Z(\la)$,
$$
\pr(Z(\la)=j)=\frac{\la^j/j!}{f(\la)},\quad j\ge 2.
$$
Observe that the probability generating function (p.g.f.) of $Z(\la)$ is given by
\begin{align} \label{pgf}
\ex\bigl[z^{Z(\la)}\bigr]=\frac{f(z\la)}{f(\la)},
\end{align}
thus
\begin{align} \label{pgfE}
\ex[Z(\la)]=\left.\frac{\partial}{\partial z}\frac{f(z\la)}{f(\la)}\right|_{z=1}=
\frac{\la f'(\la)}{f(\la)}=\frac{km}{n},
\end{align}
the final equality using \eqref{lambda}, and 
\begin{multline}\label{VarZ=}
\Var[Z(\la)]=\left.\frac{\partial^2}{\partial z^2}\frac{f(z\la)}{f(\la)}+\frac{\partial}{\partial z}\frac{f(z\la)}{f(\la)}-\left[\frac{\partial}{\partial z}\frac{f(z\la)}{f(\la)}\right]^2\right|_{z=1}\\
=\frac{\la^2 f''(\la)}{f(\la)}+\frac{\la
f'(\la)}{f(\la)}-\left[\frac{\la f'(\la)}{f(\la)}\right]^2=\Theta(\la)
\end{multline}
(specifically, for $\la>0$, $\tfrac13 \la \leq \Var[Z(\la)] \leq \la$).
\si

With these preliminaries done,
we focus on asymptotics of $\card\AAmn$, $\card\BBmn$ and $\card\CCmn$.

\begin{lemma}\label{LA}
Suppose \mnhypi. 
Then, with $\la$ as in \eqref{lambda}, 
\begin{align}
\card\CCmn=&\,\frac{1+O(n^{-1})}{\sqrt{2\pi n\Var[Z(\la)]}}\,(km)!\frac{f(\la)^n}{\la^{km}},
\label{C}\\
\frac{\card\BBmn}{\card\CCmn}=&\,\exp\left(-\frac{k-1}{2}\,\frac{\la e^{\la}}
{e^{\la}-1}\right)+o(1),\label{B}
\end{align}
so that the fraction $\card\BBmn/\card\CCmn$ is bounded away from zero.
Consequently
\begin{align}
\card\AAmn
= \frac{\card\BBmn}{(k!)^m}
=&\,\frac{1+o(1)}{\sqrt{2\pi n\Var[Z(\la)]}}\,\frac{(km)!}{(k!)^m}
\,\frac{f(\la)^n}{\la^{km}}
\exp\left(-\frac{k-1}{2}\,\frac{\la e^{\la}}{e^{\la}-1}\right).\label{A}
\end{align}
\end{lemma}
\bi

\begin{corollary} \label{preq}
Under the hypotheses of Lemma~\ref{LA},
uniformly for all non-negative, matrix-dependent functions~$r$,
\begin{equation*} 
\ex[r(A)]=
\ex[r(B)]=O(\ex[r(C)]).
\end{equation*}
\end{corollary}

\begin{proof}
The first equality is trivial.
To show the second,
for any $\mathcal{H}\subseteq \BBmn$,
\begin{align}   \label{P(B)=O(P(C))}
\pr(B\in \mathcal{H})
 & =\pr(C\in \mathcal{H} \cond C\in \BBmn)
 \notag \\ &= \frac{\pr(C\in \mathcal{H},\,C\in \BBmn)}
{\card\BBmn \,/\,\card\CCmn }
\leq\,  \frac{\card\CCmn}{\card\BBmn} \,  {\pr(C\in\mathcal{H})}
=O(1) \, \pr(C\in\mathcal{H})
\end{align}
by \eqref{B}.
\end{proof}

\begin{proof}[Proof of Lemma~\ref{LA}]
Equation \eqref{A} is immediate from \eqref{C} and \eqref{B}.
Proving \eqref{C} and \eqref{B}
will occupy the rest of this section.

\textbf{We first prove \eqref{C}.} To determine $\card\CCmn$, recall that each row
$i \in m$ is
given its own $k$, mutually distinguishable, chips.
We can get an allocation $C \in \CCmn$ by permuting all the chips
and allocating the first $j_1 \geq 2$ chips to column~1,
the next $j_2 \geq 2$ chips to column~2, \textit{etc.};
each chip goes to its predetermined row and its random column.
Up to the irrelevant permutation of chips within the first $j_1$,
the next $j_2$, \textit{etc.},
an allocation $C \in \CCmn$ is uniquely determined by such a scheme.

We adopt the notational convention that for
$h(z)=\sum_jh_jz^j$, 
$ [z^j]\,h(z) := h_j
$.
We thus have
\begin{align} 
\card\CCmn=&
 \,\sum_{j_1+\cdots+j_n=km\atop j_1,\dots,j_n\ge 2}
    \frac{(km)!}{j_1!\cdots j_n!}
 \notag \\
=&\,(km)!\,[z^{km}]\left(\sum_{j\ge 2}\frac{z^j}{j!}\right)^n
 =(km)!\,[z^{km}]f(z)^n
 \label{14.1}
 \\
=&\,(km)!\frac{f(\la)^n}{\la^{km}}\,[z^{km}]\left(\frac{f(z\la)}{f(\la)}\right)^n
 \notag \\
=&(km)!\frac{f(\la)^n}{\la^{km}}\,[z^{km}]\bigl(\ex[z^{Z(\la)}]\bigr)^n
 \quad \text{(see \eqref{pgf})}
 \notag \\
=&\,(km)!\frac{f(\la)^n}{\la^{km}}\,\pr\left(\sum_{j=1}^nZ_j(\la)=km\right),
 \label{|Cmn|}
\end{align}
where $Z_1(\la),\dots,Z_n(\la)$ are independent copies of $Z(\la)$.
Now, since
$\Var[Z(\la)]=\Theta(\la)$ 
(by \eqref{VarZ=}) and $\liminf\la>0$ (by $\la=\la(km/n)$ and the hypothesis
that $\lim m/n > 2/k$), 
we have
$\liminf\Var[Z(\la)]>0$.
So, by a local limit theorem
(Aronson, Frieze and Pittel~\cite[equation (5)]{ArFrPi}), 
$$
\pr\left(\sum_{j=1}^nZ_j(\la)=km\right)=\pr\left(\sum_{j=1}^nZ_j(\la)=n\ex[Z(\la)]\right)=
\frac{1+O(n^{-1})}{\sqrt{2\pi n\mathrm{Var}[Z(\la)]}},
$$
which proves \eqref{C}.

\textbf{We now prove \eqref{B}.} Let $C=\{c_{i,j}\}$ be distributed uniformly on $\CCmn$. Let $M$ denote
the number of cells that house $2$ or more chips, i.e.,
$
M=\bigl|\{(i,j) \colon c_{i,j}\ge 2\}\bigr|.
$
Let
$\MM$ be the number of pairs of chips hosted by the same cell, i.e.,
$$
\MM=\sum_{(i,j) \colon c_{i,j}\ge 2}\binom{c_{i,j}}{2}=\sum_{(i,j)}\binom{c_{i,j}}{2}.
$$
$\MM=M$ iff there are no cells hosting more than $2$ chips.
Clearly
$$
\frac{\card\BBmn}{\card\CCmn}=\pr(C\in\BBmn)=\pr(\MM=0).
$$
Of course, $\pr(\MM=0)=\pr(M=0)$, but, unlike $M$, $\MM$ is amenable to moment
calculations.
\si

Denoting the indicator of an event $E$ by $\one(E)$, we write
\begin{align}
\MM &=\sum_{i\in [m],\,j\in [n]}\sum_{1\le u<v\le k}\one\bigl(E(i,j;u,v)\bigr),
 \label{MM}
\end{align}
where $E(i,j;u,v)$ is the event that, of the $k$ chips owned by row $i$, at
least the two chips
$u$ and $v$ were put into cell $(i,j)$. Each of these $mn\binom{k}{2}$ event
indicators has the same expected value,
\begin{equation}\label{Exp}
\ex[\one(E(i,j;u,v))] =
(km-2)!\,\frac{[x^{km-2}]f(x)^{n-1}e^x}{\card\CCmn} .
\end{equation}
To see why \eqref{Exp} is so, compare with \eqref{14.1}
and note that once we have put two selected chips into a cell $(i,j)$
we allocate the remaining $(km-2)$ chips
amongst $n$ columns, at least two per column,
with the exception (hence the sole $e^x$ factor)
that the $j$th column receives an unconstrained
number of additional chips (as it already has two).
Arguing as for \eqref{|Cmn|},
\begin{align}\label{Arg}
[x^{km-2}]f(x)^{n-1}e^x
 &=
 \frac{f(\la)^{n-1}e^{\la}}{\la^{km-2}} \; 
  \Pr{\sum_{j=1}^{n-1}Z_j(\la)+ \Pola=km-2} ,
\end{align}
where $\Pola$ stands for an independent, 
usual (not truncated) Poisson$(\la)$ random variable. 
This last probability equals
\begin{align*}
 \sum_r \; \Big[
  \Pr{ \Po(\l)=r } \: \cdot \:
  \pr \Big( \sum_{j=1}^{n-1} Z_j(\l)= km-2-r \Big)
 \Big]
\end{align*}
By the local limit theorem for $\sum_{j=1}^{n-1} Z_j(\l)$,
for $r \leq \ln n$
the second probability in the $r$th term of the sum is again
asymptotic to $(2\pi n \Var[Z(\l)])^{-1/2}$.
Then so is the probability in \eqref{Arg},
since $\Pr{\Pola > \ln n} = O(n^{-K})$, for every $K>0$.
From this, \eqref{MM}, \eqref{Exp}, \eqref{Arg}, and \eqref{C}, 
$$
\ex[\MM] 
 =(1+o(1))\frac{mn\binom{k}{2}}{(km)_2}\,\frac{\la^2e^{\la}}{f(\la)}
$$
with the usual falling-factorial notation
$(a)_b:=a(a-1)\cdots(a-b+1)$.  
Recalling \eqref{lambda} 
and setting
\begin{equation} \label{gammaDef}
\gamma := \frac{k-1}{2}\,\frac{\la e^{\la}}{e^{\la}-1}
\end{equation}
gives
$$
\ex[\MM] 
 =\gamma+o(1).
$$

More generally, we now show that for every fixed $t\ge 1$ we have
\begin{equation}\label{ENt}
 \ex[(\MM)_t] = \gamma^t +o(1).
\end{equation}
Letting $\ii=(i_1,\dots,i_t)$, $\jj=(j_1,\dots,j_t)$,
$\uu=(u_1,\dots u_t)$, $\vv=(v_1,\dots,v_t)$, we have
$$
(\MM)_t=\sum_{\ii\in [m]^t,\,\jj\in [n]^t}\,\,\sum_{\uu<\vv}
\one\left(\bigcap_{s=1}^t E(i_t,j_t;u_t,v_t)\right).
$$
Hence
$$
\ex[(\MM)_t]=\sum_{\ii\in [m]^t,\,\jj\in [n]^t}\,\,\sum_{\uu<\vv}
\pr\left(\bigcap_{s=1}^t E(i_t,j_t;u_t,v_t)\right).
$$
We break the sum into two parts, $\Sigma_1$ and the remainder $\Sigma_2$, where
$\Sigma_1$ is the restriction to
$\ii$ and $\jj$ each having all its components distinct. In
$\Sigma_1$ the number of summands is $(m)_t(n)_t\binom{k}{2}^t$, and each
summand is
$$
(km-2t)!\,\frac{[x^{km-2t}]\,f(x)^{n-t}(e^x)^t}{\card\CCmn};
$$
see the explanation following \eqref{Exp}. Analogously to \eqref{Arg},
\begin{equation*}
 [x^{km-2t}]f(x)^{n-t}(e^x)^t
 =\frac{f(\la)^{n-t}(e^{\la})^t}{\la^{km-2t}} 
  \, \Pr{ \; \sum_{j=1}^{n-t}Z_j(\la)+\sum_{s=1}^t \Posla = km-2t},
\end{equation*}
where the $\Posla$ are $\Po(\la)$ random variables 
independent of one another and of
$Z_1(\la),\dots, Z_{n-t}(\la)$. As before, the probability is
asymptotic to $\bigl(2\pi n\Var[Z(\la)]\bigr)^{-1/2}$. So, using \eqref{C} and
recalling \eqref{gammaDef}, we have
\begin{align}
\Sigma_1\sim&\,\frac{(m)_t(n)_t\binom{k}{2}^t}{(km)_{2t}}\,\left(\frac{\la^2e^{\la}}{f(\la)}
\right)^t\notag\\
\sim&\,\left[\frac{mn\binom{k}{2}}{(km)^2}\,\frac{\la^2e^{\la}}{f(\la)}\right]^t\to\gamma^t.
\label{Sig1}
\end{align}
In the case of $\Sigma_2$, 
letting $I=\{i_1,\dots,i_t\}$, $J=\{j_1,\dots,j_t\}$, we have
$|I|+|J|\le 2t-1$. So the number of attendant pairs $(I,J)$ is at most $(m+n)^{2t-1}
=O\bigl(m^{2t-1}\bigr)$. 
The number of pairs $(\ii,\jj)$ inducing a given pair $(I,J)$ is bounded above by a constant $s(t)$. 
For every one of those $s(t)$ choices, we select pairs of chips for each of
the chosen $t$ cells; there are at most $\binom{k}{2}^t$ ways of doing so.
Lastly, we allocate the remaining $(km-2t)$ chips in such a way that every
column $j\in [n]\setminus J$ gets at least $2$ chips. 
As in the case of $\Sigma_1$, this can be done in
\begin{multline*}
(km-2t)!\,[x^{km-2t}]\,f(x)^{n-|J|}(e^x)^{|J|}\\
=
(km-2t)!\,
\frac{f(\la)^{n-|J|}(e^{\la})^{|J|}}{\la^{km-2t}}\,\,\pr\!\!\left(\sum_{j=1}^{n-|J|}Z_j(\la)+
\sum_{s=1}^{|J|}\Po_s(\la)=km-2t\right)
\end{multline*}
ways. 
Again, the probability is asymptotic to $\bigl(2\pi n\var[Z(\la)]\bigr)^{-1/2}$. So,
as $e^{\la}>f(\la)$, the sum $\Sigma_2$ is of order
\begin{equation}\label{Sig2}
m^{2t-1}\,\frac{(km-2t)!}{(km)!}\left(\frac{e^{\la}\la^2}{f(\la)}\right)^{t}=O(m^{2t-1}/m^{2t})=O(m^{-1}).
\end{equation}
Combining \eqref{Sig1} and \eqref{Sig2}, and recalling \eqref{gammaDef}, we
conclude that for each fixed $t\ge 1$,
$$
\ex[(\MM)_t] = \gamma^t+o(1).
$$
Therefore $\MM$ is asymptotic, with all its moments and in distribution, to $\Po(\gamma)$.
In particular,
$$
\pr(\MM=0) =\pr(\Po(\gamma)=0)+o(1)=e^{-\gamma}+o(1).
$$
This completes the proof of Lemma~\ref{LA}.
\end{proof}
\bi

\section{Counting critical row subsets, and the main result}
\label{MainResult}

This section will prove Theorem~\ref{main}.
Remark~\ref{>1} already dealt with the case $\lim m/n>1$.
It suffices, then, to show that with $\lim m/n \in (2/k,1)$,
the expected number of nonempty critical row sets goes to 0:
then with high probability there is no such set,
$A$ is of full rank, and the instance is satisfiable.

In the model $\CCmn$,
Lemma~\ref{expbounds} gives an upper bound on
the expected number of critical row sets of each cardinality 
$\ell \in \set{1,\ldots,m}$  
as a function of $c=m/n$, $k$, $n$, and $\ell$, 
minimized over two additional variables $\z_1$ and $\z_2$.
Lemma~\ref{Hlemma} shows that, for $c \in (2/k,1)$,
there exist values for $\z_1$ and $\z_2$ 
making this bound small,
in particular making its exponential dependence on $n$
decreasing rather than increasing.
Corollary~\ref{dupCor} uses Lemma~\ref{Hlemma} to show that
in the model $\AAmn$
the total expected number of nonempty critical row sets is of order
$O\bigl(m^{-(k-2)}\bigr)$,
proving Theorem~\ref{main}.

Lemma~\ref{Hlemma} is established by several claims deferred to 
Section~\ref{ProveLemma},
and Section~\ref{Unconstrained} extends Theorem~\ref{main} to the 
unconstrained $k$-XORSAT model (Theorem~\ref{unconstrained}).

\begin{lemma}\label{expbounds}  
Suppose $k\ge 3$ and \mnhypi,
and let $C$ be chosen uniformly at random 
from $\CCmn$.
For $\ell \in \set{1,\ldots,m}$,
let $\Ymnl$ denote the number of 
critical row sets of $C$ of cardinality~$\ell$.
Then, with $c=m/n$,
$\l=\l(ck)$ as given by \eqref{lambda},
introducing $\zzz=(\z_1,\z_2)>\zero$ and letting $\ab=1-\a$,
\begin{equation}\label{YlesseH}
\ex\bigl[\Ymnl\bigr]
  \leb \,
 \sqrt{\tfrac{1}{\z_2}}\,
 \exp\bigl[n H_k(\alpha,\zzz;c)\bigr],
\quad\forall\,\zzz>0,
\end{equation}
where
\begin{multline}\label{H=}
H_k(\alpha,\zzz;c)=c H(\alpha)
+ck\a \ln(\a/\z_1) +ck\ab \ln(\ab/\z_2)
\\* 
 +\ln\frac{f(\l(\z_2+\z_1))+f(\l(\z_2-\z_1))}{2f(\la)} ,
\end{multline}
by continuity we define $x \ln x=0$ at $x=0$,
and $H(\a)$ is the usual entropy function
$$
H(\alpha):=
-\alpha\ln\alpha-(1-\alpha)\ln({1-\alpha}) .
$$
\end{lemma}

\begin{proof}
By symmetry,
\begin{equation}\label{Ymnl}
\ex[\Ymnl]=\binom{m}{\ell}\pr(\DD_{\ell});\quad
\DD_{\ell}:=\bigcap_{j=1}^{n}\left\{\sum_{i=1}^{\ell}c_{i,j}
\text{ is even}\right\}.
\end{equation}
By symmetry again,
\begin{align}\label{nuth}
\pr(\DD_{\ell})
=&\,\sum_{\nu=1}^n\binom{n}{\nu} \pr(\DD_{\ell,\nu}), 
\end{align}
where
\begin{align}
 \DD_{\ell,\nu}:=&\,
\bigcap_{j=1}^{\nu}\left\{\sum_{i=1}^{\ell}c_{i,j}
\text{ is even, positive}\right\}\bigcap\bigcap_{j=\nu+1}^n\left\{\sum_{i=1}^{\ell}c_{i,j}=0\right\}.
\end{align}
Recalling that $\sum_{i\in [m]}c_{i,j}\ge 2$, 
we see that on the event $\DD_{\ell,\nu}$,
\begin{equation}\label{sepcon}
\sum_{i\le\ell}c_{i,j}=\left\{\begin{alignedat}2
&\text{even }>0,\quad&&j\le\nu,\\
&0,\quad&&j>\nu;\end{alignedat}\right.\qquad
\sum_{i>\ell}c_{i,j}\geq\left\{\begin{alignedat}2
&0,\quad&&j\le\nu,\\
&2,\quad&&j>\nu.\end{alignedat}\right.
\end{equation}
Thus on $\DD_{\ell,\nu}$ the column sums of the two complementary
submatrices, $\{c_{i,j}\}_{i\leq \ell,j\in [n]}$ and $\{c_{i,j}\}_{i>\ell,j\in [n]}$, are subject
to independent constraints.

Let $\CCmn(\ell,\nu)$ denote the set of all matrices $C$ with
row sums $k$ which meet the constraints \eqref{sepcon}.  
Then $\pr(\DD_{\ell,\nu})$ is given by
\begin{equation}\label{plnuless}
p(\ell,\nu):=\pr(\DD_{\ell,\nu})=\frac{|\CCmn(\ell,\nu)|}{\card\CCmn}.
\end{equation}
By the independence of constraints on column sums for the upper and the lower
submatrices of the matrices $C$ in question,
\begin{equation} \label{c=ab}
|\CCmn(\ell,\nu)|=\alnu \cdot \bmnu,
\end{equation}
where (paralleling our definition of $\CCmn$ in Section~\ref{PrSpaces})
$\alnu $ is the number of ways to assign
$k\ell$ chips among the first $\nu$ columns so that each of those columns gets a positive even number of chips,
and 
$\bmnu$ is the number of ways to assign $k(m-\ell)$ chips among
all $n$ columns so that each of the last $(n-\nu)$ columns gets at least $2$ chips.

As in \eqref{|Cmn|},
\begin{align}
\alnu 
&=
\sum_{j_1+\cdots+j_{\nu}=k\ell\atop j_s>0,\text{ even}}
\frac{(k\ell)!}{j_1!\cdots j_{\nu}!}
\notag \\ &= 
(k\ell)!\,[z^{k\ell}]\left(\sum_{j>0,\text{ even}}\frac{z^j}{j!}\right)^{\nu}
\notag \\&
=(k\ell)!\,[z^{k\ell}](\cosh z -1)^{\nu},  \label{aseries}
\end{align}
and
\begin{align}
\bmnu=&\,\sum_{j_1+\cdots+j_n=k(m-\ell)\atop
j_1,\dots,j_{\nu}\ge 0;\,\, j_{\nu+1},\dots,j_n\ge 2}
\frac{(k(m-\ell))!}{j_1!\cdots j_n!}
\notag \\ =&
(k(m-\ell))!\,[z^{k(m-\ell)}] (e^z)^{\nu} f(z)^{n-\nu}.   \label{bseries}
\end{align}
Since the coefficients of the Taylor expansion around $z=0$ of 
$e^{z\nu}f(z)^{n-\nu}$
are non-negative, we use these identities in a standard (Chernoff) way to bound
\begin{equation}\label{abless}
\begin{aligned}
\alnu \le&\,(k\ell)!\,\frac{(\cosh z_1 -1)^{\nu}}{z_1^{k\ell}},\quad
\forall\, z_1>0 .
\end{aligned}
\end{equation}
We could bound $\bmnu$ similarly, but
we need a stronger bound, namely
\begin{equation}\label{blessbetter}
\bmnu\leb (nz_2)^{-1/2}(k(m-\ell))!\,\frac{(e^{z_2})^{\nu}f(z_2)^{n-\nu}}{z_2^{k(m-\ell)}},
\quad\forall\,z_2>0.
\end{equation}
The bound \eqref{blessbetter}
follows from three components:
the Cauchy integral formula
$$
\bmnu
 =
 \frac{(k(m-\ell))!}{2\pi}\!\!
 \oint\limits_{{\mbox{\Large${z=z_2e^{i\theta} \colon 
     \atop\theta\in (-\pi,\pi]}$}}\!\!}
 \frac {(e^{z})^{\nu}f(z)^{n-\nu}} {z^{k(m-\ell) +1 }} \, dz,
$$
and (with $z=z_2 e^{i\theta}$) 
the identity
$|e^z|=e^{z_2}\exp\bigl[-z_2(1-\cos\theta)\bigr]$ 
and the less obvious inequality
\begin{align} 
 |f(z)| & \le |f(z_2)|\exp\bigl[-z_2(1-\cos\theta)/3\bigr].
 \label{PittelIneq}
\end{align} 
(See Pittel~\cite[Appendix]{Pittel86} for the inequality,
and Aronson, Frieze and Pittel~\cite[inequality (A2)]{ArFrPi} 
for how it works in combination with the Cauchy formula.)

Using 
\eqref{plnuless}, 
\eqref{c=ab}, \eqref{abless}, \eqref{blessbetter},
with $\card\CCmn$ from
\eqref{C} 
and $\Var{Z(\la)}$ from \eqref{VarZ=}, we obtain
that, $\forall\,z_1,z_2>0$,
\begin{equation}
p(\ell,\nu)\leb \sqrt{\frac{\la}{z_2}}\,\binom{km}{k\ell}^{-1}\frac{\la^{km}}{z_1^{k\ell}\,z_2^{k(m-\ell)}}\,\,
\frac{[e^{z_2}(\cosh z_1-1)]^{\nu}f(z_2)^{n-\nu}}{f(\la)^n} .
\label{peq}
\end{equation}

Now, it is immediate from
\eqref{Ymnl}, \eqref{nuth}, and \eqref{plnuless} that
\begin{align}
\ex\bigl[\Ymnl\bigr]
 = \binom m \ell \sum_{\nu=1}^n \binom n \nu p(\ell,\nu) .
 \label{EYmn1}
\end{align}
If we restrict to $z_1$ and $z_2$ depending only on $\ell$, $m$ and $n$
(not on $\nu$),
then on substituting \eqref{peq} into the above
we may simplify the sum to obtain
\begin{align}
\ex\bigl[\Ymnl\bigr]
 \leb 
 &\, \sqrt{\frac{\la}{z_2}}\, \binom{m}{\ell}\binom{km}{k\ell}^{-1} \la^{km}
  \notag\\*
 \times&\,\frac{1}{z_1^{k\ell}\,z_2^{k(m-\ell)}}
 \parens{ \frac{f(z_2)+e^{z_2}(\cosh z_1-1)}{f(\la)}}^n
 ,\quad\forall\,z_1,z_2>0 \label{Ymnlless}
.
\end{align}
Observe that
$$
f(z_2)+e^{z_2}(\cosh z_1-1)=\frac{f(z_1+z_2)+f(z_2-z_1)}{2}.
$$
Inequality \eqref{YlesseH}, and thus the lemma,
are established by substituting this
and the Stirling-based approximation 
$\binom n {pn}=O(1) \frac1{\sqrt{ {n p(1-p)}}} \exp(n H(p))$
into \eqref{Ymnlless}, recalling that
$m=cn$, 
$\alpha=\ell/m$ and $\ab=1-\a$,
substituting $z_1 = \z_1 \l$ and $z_2 = \z_2 \l$, 
and observing that $\sqrt k = O(1)$.
For $\ell=m$ the Stirling-based approximation is inapplicable
but consistency of \eqref{YlesseH} with \eqref{Ymnlless}
is easily checked.
\end{proof}

Recall the definition of $\HH$ from \eqref{H=}.
Roughly speaking,
the following lemma
establishes the existence of $\zzz$ making $\HH$ negative.
An intuitive description of the behavior of 
$\HH$ is given at the start of the next section.

\begin{lemma} \label{Hlemma}
Let
\begin{align}
\a_k &= e k^{-k/(k-2)} .   \label{ak}
\end{align}
For all 
$k \geq 4$
and $c \in (2/k,1)$,
there exist 
$\e,\zn>0$
such that
\begin{align}
 \big( \forall \a \in (0,\a_k] \, \big) \, (\exists \zzz)
  & \colon
 \HH \leq (c \a)(\tfrac k 2-1) \ln ( \a / \a_k)
  \text{ and } \z_2>\zn  \hspace*{-2cm}
  \label{lemma1.2}
\\
 \big( \forall \a \in [\dak,1] \, \big) \, (\exists \zzz)
  & \colon
  \HH \leq -\e \text{ and } \z_2>\zn  .
  \label{lemma1.1}
\end{align}
\end{lemma}

\begin{proof}
The lemma follows immediately from 
Claims~\ref{asmall}, \ref{amed}, \ref{alarge} and~\ref{ahalf},
respectively treating $\a$ in the four ranges
$(0, 0.99 \a_k]$,
$[0.99 \a_k,\medup]$,
$(\medup,1/2]$,
and
$(1/2,1)$.
A suitable function $\zzz$ is given explicitly in each case.
\end{proof}

The lemma yields the following corollary.
\begin{corollary} \label{dupCor}
Under the hypotheses of Lemma \ref{expbounds} but with $k \geq 4$,
$$
\sum_{\ell=2}^{m}\ex\bigl[\Ymnl\bigr]
=O\bigl(m^{-(k-2)}\bigr).
$$
\end{corollary}

\begin{proof}
Since $\lim m/n \in (2/k,1)$, there exists 
a closed interval $I \subset (2/k,1)$ 
such that,
for all but finitely many cases,
$c=m/n \in I$. 
Where $\e(c,k)$ and $\zn(c,k)$ satisfy the conditions of Lemma~\ref{Hlemma},
define $\e = \e(I) = \min\set{\e(c,k) \colon c \in I}$
and $\zn = \zn(I)$ likewise.
Then, for all but finitely many pairs $m,n$, 
inequalities \eqref{lemma1.2} and \eqref{lemma1.1} hold true.

Letting $\ell_k = \a_k m = \Theta(n)$,
for $\ell \leq \ell_k/2$,
recalling that $\a c n = \a m = \ell$,
\eqref{YlesseH} and \eqref{lemma1.2} give
\begin{align} \label{ymnlBound}
\ex\bigl[\Ymnl\bigr]
 =O(1) \,
 \exp\bigl[(\tfrac k2-1) \ell \ln(\ell/\ell_k)
  \bigr] ,
\end{align}
where we have incorporated $\sqrt{\textfrac1 \zn}$ in the leading $O(1)$.
By convexity of
$\ell \ln(\ell/\ell_k)$,
interpolating
for $\ell \in [2, \ell_k/2]$
from the endpoints of this interval,
\begin{align}
\ell \ln(\ell/\ell_k)
 & \leq
 2 \ln(2/\ell_k)
  + \frac{\ell-2}{\ell_k/2-2} \parens{(\ell_k/2) \ln(1/2) - 2 \ln(2/\ell_k)}
  \notag
 \\ &=
 2 \ln(2/\ell_k) + (\ell-2) (-\ln 2 + o(1)) ,
 \notag
 \\ & \leq 2 \ln (2/\ell_k) - 0.6(\ell-2)
 \notag
\end{align}
for $n$ sufficiently large, where we have used that $\ell_k = \Theta(n)$
and $0.6 < \ln 2$.
Thus, 
\begin{align*}
\ex\bigl[\Ymnl\bigr]
 & \leq
 O(1) \, \exp \big( (\tfrac k2-1) [2  \ln (2/\ell_k) -0.6 (\ell-2) ] \big)
 \\ & =
 O(1) \, m^{-(k-2)} \exp(-0.6 (\tfrac k2-1) (\ell-2)) ,
\end{align*}
where the last line 
incorporates $(2/\a_k)^{k-2}$
in the $O(1)$.
Given this upper bound that is geometrically decreasing in $\ell$,
summing gives
$$
 \sum_{\ell=2}^{\floor{(\a_k/2) m}}
  \ex\bigl[\Ymnl\bigr]
  = O \bigl(m^{-(k-2)}\bigr) .
$$

For $\ell > (\a_k/2) m$, by \eqref{lemma1.1},
$ \ex\bigl[\Ymnl\bigr]
 =O(1) \,
 \exp(-\e n),
$
giving
$$\sum_{\ell=\ceil{(\a_k/2) m}}^{m}
\ex\bigl[\Ymnl\bigr]
  = O(m) \exp(-\e n) 
  = \exp(-\Omega(n)) .
$$
Adding the two partial sums yields Corollary~\ref{dupCor}.
\end{proof}

\begin{proof}[Proof of Theorem~\ref{main}]
By the remarks at the start of this section, 
we need only consider the case $\lim m/n \in (2/k,1)$.
Under the hypotheses of Corollary~\ref{dupCor},
let $A \in \AAmn$ and $C \in \CCmn$ be uniformly random,
and let $\Xmn$ and $\Ymn$ denote the numbers of nonempty critical row sets of 
$A$ and $C$ respectively,
and $\Xmnl$ and $\Ymnl$ those of cardinality $\ell$.
$\Xmn^{(1)}=0$ since every row of $A$ has $k$ 1's.
($\Ymn^{(1)}$ is not necessarily $0$ since a row of $C$ can be 0,
for example if all the 1's in its defining configuration lie in a single cell.)
Then
$$
\ex\bigl[\Xmn\bigr]
= 0+ \sum_{\ell=2}^{m}\ex\bigl[\Xmnl\bigr]
= O(1) \sum_{\ell=2}^{m}\ex\bigl[\Ymnl\bigr]
=O\bigl(m^{-(k-2)}\bigr),
$$
the last two equalities coming from Corollary~\ref{dupCor}
and Corollary~\ref{preq}.
Then $\Pr{A \text{ is not of full rank}} \leq \E{\Xmn} = O(m^{-(k-2)})$,
so with probability $1-O(m^{-(k-2)})$,
$A$ is of full rank and any system $Ax=b$ is satisfiable.
\end{proof}

\section{\texorpdfstring%
{Analysis of the function ${\HH}$ to prove Lemma~\ref{Hlemma}}%
{Analysis of the function H to prove Lemma~\ref{Hlemma}}}   
\label{ProveLemma}
Recall the notation $\ab=1-\a$ and $\zzz=(\z_1,\z_2)$ 
as well as the definition of $\HH$ from \eqref{H=}.
In this section we 
use an explicit function $\zzz=\zzzargs$,
taking different forms in different ranges of $\a$,
to establish Claims~\ref{asmall}, \ref{amed}, \ref{alarge} and~\ref{ahalf}
and thus Lemma~\ref{Hlemma}.

For intuition about $\HH$, the case $k=4$ is indicative.
Figure~\ref{H4} shows a graph of the function value against $\a$, 
for a few choices of $c$,
with $\zzz$ given by \eqref{zsmall} for small $\a$, 
and by $\zzz=(\a,\ab)$ otherwise.
Numerical experiments suggest that the optimal choice of $\zzz$
leads to qualitatively similar results, though of course without the
kinks where we change from one functional form for $\zzz$ to another.
As shown, $\HH$ tends to 0 at $\a=0$ (treated in Claim~\ref{asmall}), 
but the dependence on $c$ here is not critical:
an analog of the claim, with different parameters,
could be obtained as long as $c$ is bounded away from 0 and infinity.
At $\a=1/2$ (treated in Claim~\ref{alarge}), 
the function tends to 0 as $c$ tends to 1,
so this is where $c < 1$ is required.
For values of 
$\a$ between 0 and 1/2 but bounded away from them,
the function value is bounded away from 0 (for $c \leq 1$),
so relatively crude means suffice to treat this case
(Claim~\ref{amed}).
Function values for $\a>1/2$ 
(treated in Claim~\ref{ahalf})
are dominated
by their symmetric counterparts at $1-\a$.

Note that Lemma~\ref{Hlemma} only considers $k>3$.
The lemma does in fact extend to $k=3$, but this case
was already treated by \cite{DMfocs},
and poses additional difficulties for us.
In particular, for both choices of $\zzz$ we consider below,
taking $k=3$, $c=0.999$ and $\a=1/3$ leads to $\HH>0$.

\begin{figure}[htbp]
\begin{centering}
\ifpdf{%
\includegraphics[bb=0.0in 0.0in 10in 7.5in,width=4.0in]{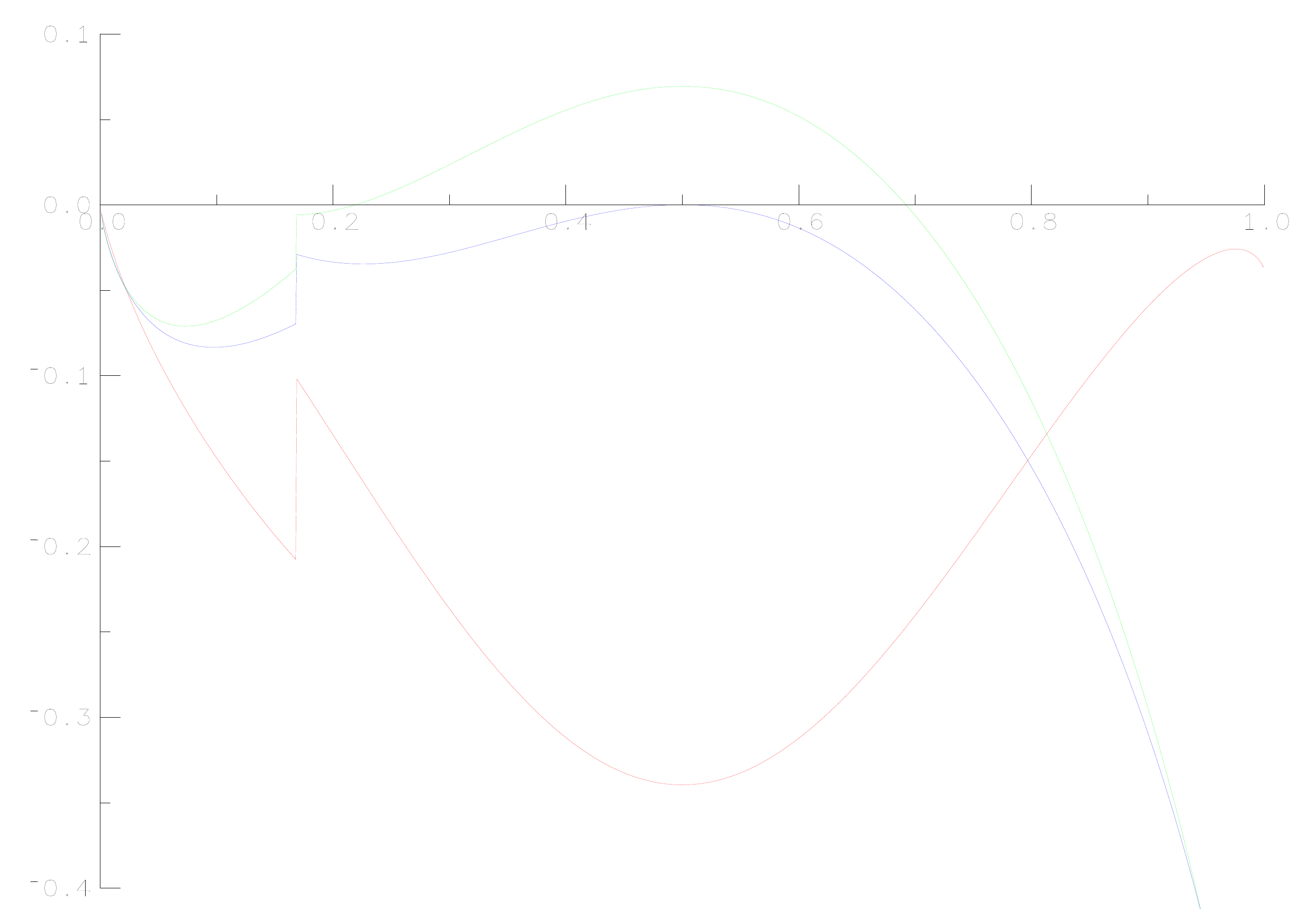}
}\else{%
\includegraphics[width=4.0in]{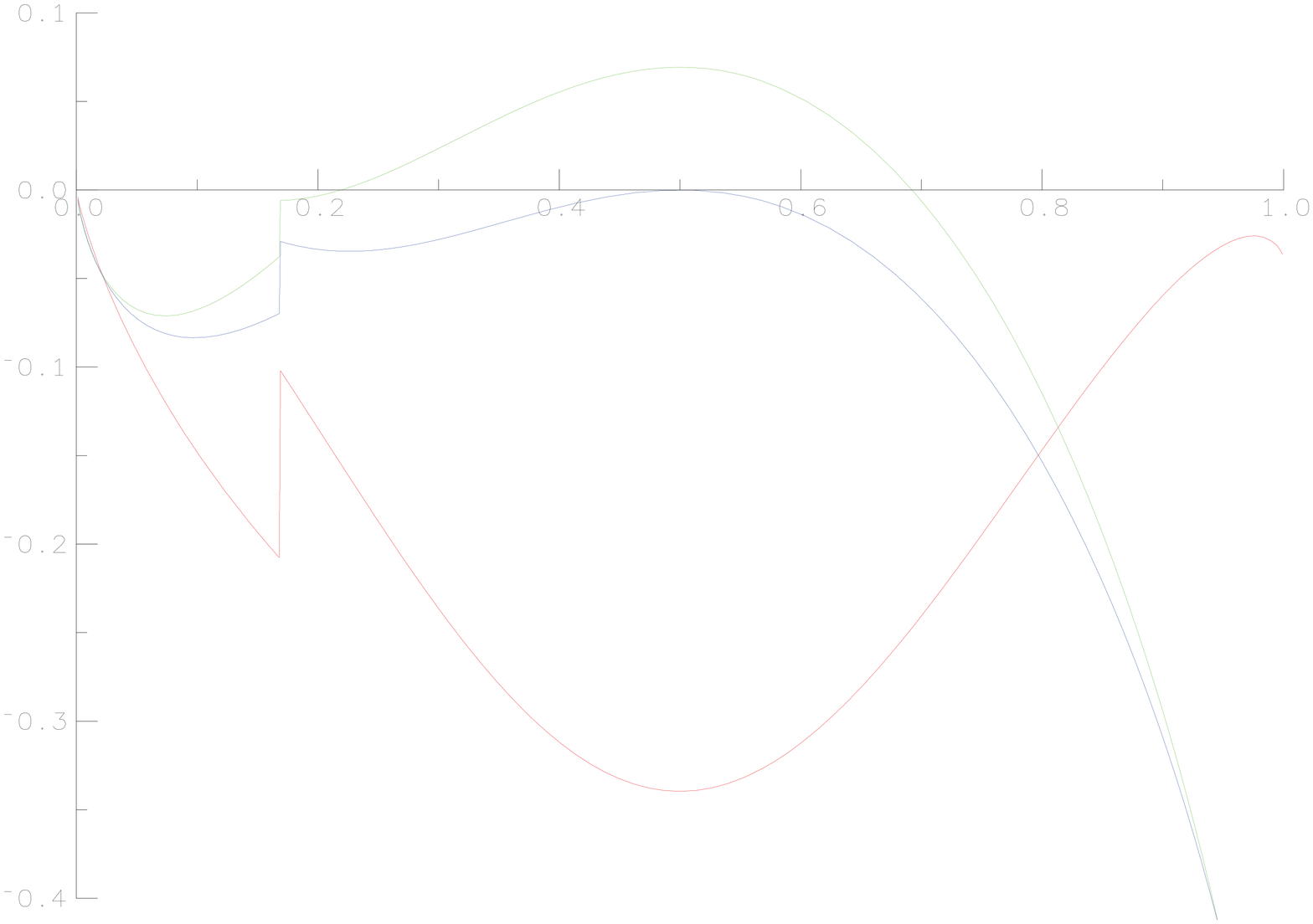}
}\fi
\end{centering}
\caption{
Plot of $\HH$ versus $\a$, for $k=4$ and $c$ values of
0.51, 1.0, 1.1 (from bottom to top).
The kinks occur at $\a_k$ where we switch functional forms for
$\zzzargs$.
}
\label{H4}
\end{figure}

\begin{claim} \label{asmall}
For all $k \geq 3$ and all $c \in (2/k, 1]$,
taking
\begin{align}
\z_1= (ck)^{-1/2} \a^{1/2}, \quad \z_2=\ab   \label{zsmall}
\end{align}
yields 
$\HH \leq (c \a)(\tfrac k 2-1) \ln ( \a / \a_k)$
for all $\a \in (0,\a_k)$.
Also, for any $\d=\d(k)>0$ 
there exists $\e=\e(k)>0$ such that
$\HH < -\e$
for all $\a \in [\d,\aak]$.
In both cases, $\z_2 \geq 1-\alpha_k>0$. 
\end{claim}

Note that the first part of the claim establishes \eqref{lemma1.2},
and the second part, with $\d=\a_k/3$, establishes \eqref{lemma1.1} 
for $\a \in [\a_k/3,0.99 \a_k]$.

\begin{proof}
Trivially,
$\z_2 = \ab \geq 1-\alpha_k>0$,
since $\a_k = e k^{-k/(k-2)} < e/k < 1$.
The issue in this range of $\a$ is 
to control the final logarithmic term of $\HH$
when the two summands within the logarithm are nearly equal.
Note that $\ln f(x)$
is concave on either side of 0
(diverging to $-\infty$ at $0$, it is not concave as a whole),
as
$$
\bigl[\ln f(x)\bigr]''=\frac{e^x(1-x-e^{-x})}{f^2(x)}<0.
$$
Since $\frac d{d\D} \ln f(\l (1+\D)) = \frac{\l f'(\l)}{f(\l)}$,
if $\l$ and $\l(1+\D)$ are on the same side of 0 (i.e., if $1+\D \geq 0$)
then concavity gives
$ \ln f(\l (1+\D))
  \leq \ln f(\l) + \D \frac{\l f'(\l)}{f(\l)} $.
Or, with $\z=1+\D$, if $\z \geq 0$ then
\begin{align}
\frac{f(\l \z)}{f(\l)}
  \leq \exp\left((\z-1) \frac{\l f'(\l)}{f(\l)}\right)
  = \exp\left((\z-1) ck \right) ,
  \label{fratioexpbound}
\end{align}
recalling from \eqref{lambda} 
that $\l f'(\l)/f(\l) = ck$.
It is easily checked that \eqref{ak} gives $\a_k < 0.2$,
hence from \eqref{zsmall}
$\z_2>0.8$ and $\z_1<0.4$, so
$\z_2-\z_1 \geq 0$ and of course
$\z_2+\z_1 \geq 0$.
Thus for the final term of $\HH$, from \eqref{fratioexpbound} we have
\begin{align*}
 \hspace*{0.5in}&\hspace*{-0.5in}
 \ln\frac{f(\l(\z_2+\z_1))+f(\l(\z_2-\z_1))}{2f(\la)}
  \\ &\leq
  \ln \left(
    \frac{ \exp(ck(\z_2+\z_1-1))} 2 +
    \frac{ \exp(ck(\z_2-\z_1-1))} 2
     \right)
  \\&=
  \ln \left(
    \exp(ck(\z_2-1)) \left[
    \frac{ \exp(ck \z_1) + \exp(-ck \z_1)}{2}
     \right] \right)
  \\&=
  ck(\z_2-1) + \ln \cosh (ck \z_1)
  \\& \leq
  ck(\z_2-1) + (ck \z_1)^2/2,
\end{align*}
using the well known inequality 
$\cosh x \leq e^{x^2/2}$ 
Now also using $-\ab \ln \ab \leq \a$ for all $\ab \in [0,1)$,
substituting $\zzz$
from \eqref{zsmall} into $\HH$,
\begin{align*}
 \HH
  & \leq -c \a \ln \a + c \a + c k \a \ln ((c k \a)^{1/2}) +0
    +ck (-\a) + \sqrt{ck\a}^2/2
  \\& = 
  c \a [(\tfrac k 2-1) \ln \a + (1-\tfrac{k}2) + \tfrac k 2 \ln (c k) ]
  \\& =
  (c \a)(\tfrac k 2-1) \ln [ \a \: \tfrac1e (ck)^{k/(k-2)} ] .
\end{align*}
Pessimistically taking $c=1$ within the logarithm
and recalling $\a_k$ from \eqref{ak},
\begin{align}
\HH
  & \leq
  (c \a)(\tfrac k 2-1) \ln ( \a / \a_k) .
   \label{small1}
\end{align}
(A different upper bound for $c$ would simply call for a different
value for $\a_k$.)
This proves the first part of the claim.

Clearly, for all $\a \in (0,\a_k)$, $\a \ln(\a/\a_k)$ 
is negative,
so for any $\d=\d(k)>0$, over 
$\a \in [\d,0.99 \a_k]$ 
it is bounded away from 0.
By hypothesis, $c \geq 2/k$ (any positive constant would do),
thus $\HH$ is also bounded away from 0, i.e., there is some $\e=\e(k)>0$
for which $\HH \leq -\e$.
This proves the second part of the claim.
\end{proof}

\begin{claim} \label{amed}
For all $k \geq 4$ and all $c \in (2/k, 1]$,
there exist $\e,\zn>0$ such that
for all $\a \in (\aak, \medup]$,
taking 
$\z_1=\a$, $\z_2=\ab$
yields
$\HH < -\e$ and $\z_2 \geq \zn$.
\end{claim}

\begin{proof}
Taking $\zn=1-\medup$, $\z_2 = \ab \geq \zn >0$ is immediate.
Let us confirm, though, that $\a_k<\medup$ so that the 
interval in the hypothesis is sensible.
For this, $\a_k=e k^{-k/(k-2)} < e/k$ 
suffices for $k \geq 10$,
and the cases $k \in \set{4,\ldots,9}$ are easily checked.

With $\z_1=\a, \z_2=\ab$, $\HH$ simplifies to 
\begin{align}
 \HH
  &=
 cH(\a) + \ln \parens{\frac12 + \frac12 \frac {f(\l(1-2\a))} {f(\l)}}
 \label{Hsimple} . 
\end{align}
Applying \eqref{fratioexpbound} and $1-2\a>0$, it follows that
\begin{align}
 \HH
  & \leq
 cH(\a) + \ln \parens{\tfrac12 + \tfrac12 \exp(-2ck\a)}
  \notag
  \\ 
 & \leq
 c[H(\a) + \ln \parens{\tfrac12 + \tfrac12 \exp(-2k\a)}] ,
  \label{Hmed}
\end{align}
the last step following from convexity of
$\ln(\tfrac12 + \tfrac12 \exp(-2ck\a)) =: \fa(c)$
as a function of $c$.
The application of convexity is simply
$\fa(c) \leq (1-c) \fa(0)+c \, \fa(1) = c \, \fa(1)$,
since $\fa(0)=0$.
The proof of convexity is that, with $L=2k\a$,
$ \textfrac{d \fa}{dc} = - \textfrac{2L}{(\exp(cL)+1)} $
is clearly increasing with $c$.

From \eqref{Hmed} and $c \geq 2/k$, 
it suffices to prove that
\begin{align*}
\fb_k(\a)
 & :=
 H(\a) + \ln(\tfrac12 + \tfrac12 \exp(-2k\a))
\end{align*}
is negative for $\a \in [\aak,\medup]$.
(For $k=3$ this fails to hold,
not just for the approximation $\fb_k(a)$, but also for
the true value of $\HH$ with this choice of $\zzz$.)
For a fixed $k$
this can be confirmed by interval arithmetic.
Specifically, since $H(\a)$ is increasing and
$\ln(\tfrac12 + \tfrac12 \exp(-2k\a))$ is decreasing,
if $\a \in [\a',\a'']$ then
$\fb_k(\a) \leq 
 H(\a') + \ln(\tfrac12 + \tfrac12 \exp(-2k\a''))$.
Thus, it suffices to
cover the interval $[\aak,\medup]$
with subintervals $[\a',\a'']$
for each of which
$ H(\a'') + \ln(\tfrac12 + \tfrac12 \exp(-2k\a')) < 0$.

For $k=4$, 
$0.1681 < 0.99 \a_k$
and 73 intervals suffice to cover $[0.1681,\medup]$
and show that $\fb_4(\a) < -10^{-5}$.%
\footnote{%
We cannot get a significantly smaller bound since
$\fb_4(\medup) \approx -0.0000149$:
we chose $\a=\medup$ roughly as large as possible
in order to minimize the work left for Claim~\ref{alarge}.
However, the proof there can work with $\a$ as small as 0.2736
(or smaller, with minor modifications),
and using that instead of $\medup$ here
would allow us to cover with 59 subintervals
and obtain $\fb_4(\a) < -10^{-4}$.%
}

For $k=5$, $0.1840<\aak$, and the 2 intervals 
$[0.1840, 0.2291]$ and $[0.2291, 0.2743]$
suffice to show that $\fb_5(\a) < -0.005$.

For $k \geq 6$, we first claim that $\tfrac1k<0.99 \ak$.
Multiplying through by $k$, taking logarithms, then multiplying by $(k-2)/2$,
this is equivalent to
$0< \tfrac{k-2}2 \ln(0.99 e)- \ln(k)$.
This is true for $k=6$, and
true for all larger $k$
since the derivative of the right hand side is positive.
We now prove by induction on $k$ that
$\fb_k(\a)<0$ over $\a \in [\frac1k,\medup]$,
for $k \geq 6$.  
For the base case $k=6$,
the previous interval arithmetic approach, 
with the two intervals $[0.1666,0.2204]$, $[0.2204, \medup]$,
establishes that $\fb_6(\a) < -0.03$.
Since $\fb_k(\a)$ is monotone decreasing in $k$,
$\fb_k(\a) \leq \fb_{k-1}(\a) <0$ for $\a \in [\frac1{k-1},\medup]$,
by the inductive hypothesis,
so we need only show that
$\fb_k(\a) <0$ for $\a \in [\frac1k,\frac1{k-1}]$.
Over this interval,
$H(\a) \leq H(\frac1{k-1}) \leq H(\frac16) < 0.451$,
while the other term of $\fb_k(\a)$ 
is decreasing in $k\a$,
and $k\a \geq 1$, so
 $\ln(\tfrac12 + \tfrac12 \exp(-2k\a)) \leq
   \ln(\tfrac12 + \tfrac12 \exp(-2)) < -0.566$;
summing the two terms proves that $\fb_k(\a)<-0.1$.
\end{proof}

\begin{claim} \label{alarge}
For all $k \geq 4$ and all $c \in (2/k, 1)$,
there exist $\e = \e(c,k)>0$ and $\zn = \zn(c,k) >0$ such that
for all $\a \in (\medup, 1/2]$,
taking $\z_1=\a$, $\z_2=\ab$ yields
$\HH < -\e$ and $\z_2 \geq \zn$.
\end{claim}

\begin{proof}
Again, $\z_2 \geq \zn=1/2$ is immediate.
In this case, with $\a$ relatively close to $1/2$,
the key is to govern the term $f( (1-2\a)\l) / f(\l)$
in the expression \eqref{Hsimple} for $\HH$.
We make the substitution $x=1-2\a$.

Dealing first with the leading term of $\HH$,
we have 
\begin{align}
H(\a) = H \parens{\frac12-\frac x2} &\leq \ln \parens{\frac4{x^2+2}}
 \label{entropybound}
\end{align}
for $x \in (0,1]$.
This can be verified by checking that
$H(\tfrac12-\tfrac x2) - \ln \parens{\frac4{x^2+2}}$
and its first derivative are both 0 at $x=0$,
while the second derivative,
$-\frac {x^2(10-x^2)} { (x^2+2)^2 (1-x^2)}$,
is negative 
for $x \in (0,1)$.

Returning now to the term $f( (1-2\a)\l) / f(\l)$,
motivated by the small-$\l$ asymptotic equality $f(\l) \sim \frac12 \l^2$,
we will show that
\begin{align*}
 \rr(\l,x) := \frac{f(\l x)}{x^2 f(\l)} \leq q ,
\end{align*}
where we may choose $q=1$ for any $x \in (0,1]$ and $\l>0$,
and smaller values of $q$ for restricted ranges of $x$ and $\l$.

To establish this, we first show that
$\rr(\l,x)$
is weakly increasing with $x$.
\begin{align*}
 \frac d{dx} \ln(\rr(\l,x))
  &= \frac {\l f'(\l x)} {f(\l x)} - \frac2x
  = \frac1{x f(s)} \parens{s f'(s)-2 f(s)} ,
\end{align*}
where we define $s:=\l x$.
To show that the derivative is positive, it suffices to show that
\begin{align*} 
s f'(s)-2 f(s)
  = s(e^s-1)-2(e^s-s-1)
  = (s-2)e^s +s+2 
  \geq 0 .
\end{align*} 
Indeed, the final inequality holds for all $s>0$ because
the expression and its first derivative are both 0 at $s=0$,
while the second derivative, $s e^s$, 
is nonnegative for all $s \geq 0$.

We next show that $\rr(\l,x)$ is weakly decreasing with $\l$.
\begin{align*}
 \frac d{d\l} \ln(\rr(\l,x))
  &= - \frac {f'(\l)} {f(\l)}
     + \frac {x f'(\l x)} {f(\l x)} ,
\end{align*}
and we wish to show that this is $\leq 0$.
It is obviously 0 at $x=1$,
so it suffices to check that it is increasing with respect to $x$ for $x>0$.
Since the first term is constant and $\l$ is constant,
this is equivalent to
 ${\l x f'(\l x)} / {f(\l x)} = \psi(\l x)$
being increasing with respect to $\l x$,
which is true by Remark~\ref{psi''}.

\sloppypar
Since $R(\l,x)$ is increasing with $x$ and decreasing with $\l$,
for all $\l \geq \l_0$ and $x \in (0,x_0]$,
$R(\l,x) \leq R(\l_0,x_0)$,
and $f(\l x)/f(\l) \leq x^2 R(\l_0,x_0)$.
Since $f(0)=0$, the last formulation extends to $x=0$.
That is,
for all $\l \geq \l_0$ and $x \in [0,x_0]$,
$f(\l x)/f(\l) \leq x^2 R(\l_0,x_0)$.

Since the lemma concerns $\a \in (\medup,1/2]$
we are interested in
$x = 1-2\a \leq x_0 := 1-2 \cdot \medup = \twicelo$.
We now proceed with two cases.

The first case treats $\l \geq \l_0 = \llo$ and $x \in [0,x_0)$.
These range of $\l$ permits $c=1$, where $\a=1/2$
would give the (unacceptable) value $\HH=0$;
\emph{this is thus the crucial case in the analysis, 
relying on the strict inequality $c<1$.}
Here,
$
R(\l,x) 
 \leq R(\l_0,x_0)
 < 0.5
$.
Thus, from \eqref{Hsimple} and \eqref{entropybound},
\begin{align}
 \HH
   &=
   c H(\tfrac12-\tfrac x2)
    + \ln \parens{\tfrac12 + \tfrac12 \cdot \tfrac{f(\l x)}{f(\l)}}
   \notag \\ & \leq
   c \ln\parens{\tfrac4{x^2+2}}
    + \ln\parens{\tfrac12 + \tfrac12 \cdot 0.5 \, x^2}
    \label{keycase}
   \\ & =
   (c-1) \ln\parens{\tfrac4{x^2+2}} .
\notag \intertext{Since $c-1<0$, this is maximized when the logarithm is minimized,
and is thus}
   & \leq
   (c-1) \ln\parens{\tfrac4{x_0^2+2}} 
   \notag \\ & <
   0.59 (c-1) .
   \notag
\end{align}

The second case treats the remaining values of $\l$, namely
$0< \l < \l_0$, and $x \in [0,x_0)$.
Here, $R(\l,x) \leq \lim_{\l \to 0} R(\l,x_0) = 1$.
Recalling from \eqref{lambda} 
that $ck=\psi(\l)$, and 
from Remark~\ref{psi''} that $\psi$ is increasing,
$c=(ck)/k \leq \psi(\l_0)/k \leq \cklo/4$,
as $k \geq 4$.
Thus,
\begin{align}
 \HH
   &=
   c H(\tfrac12- \tfrac x2)
    + \ln \parens{\tfrac12 + \tfrac12 \cdot \tfrac{f(\l x)}{f(\l)}}
   \notag
   \\ &<
   \tfrac{\cklo}{4} \ln\parens{\tfrac{4}{x^2+2}}
    + \ln\parens{\tfrac12 + \tfrac12 \cdot 1 \cdot x^2} ,
    \label{entropybounduse}
   \\ &<
   -0.0011   \quad \text{ for $x \in [0,x_0)$.}
   \notag
\end{align}
The final statement can be checked by verifying that the previous
expression has nonnegative derivative for all $x \geq 0$ and, at $x=x_0$,
is $<-0.0011$.

The lemma follows, with $\e = \min\set{0.59(1-c), 0.0011}$.
\end{proof}

\begin{claim} \label{ahalf}
For all $k \geq 4$ and all $c \in (2/k, 1]$,
there exist $\e = \e(c,k)>0$ and $\zn = \zn(c,k) >0$ such that
for all $\a \in (1/2,1]$
there exists $\zzz$ for which 
$\HH < -\e$ and $\z_2 \geq \zn$.
\end{claim}

\begin{proof}
For any $x >0$, $f(x) >f(-x)$;
this follows from $f(x)-f(-x)=e^x-e^{-x}-2x=2(\sinh(x)-x) > 0$,
the last inequality well known.
Since $ck>2$, $\la=\la(ck) >0$, and \eqref{Hsimple} gives
$$
\lim_{\a \to 1} H_k(\a,(\a,\ab);c)=\ln\frac{f(\la)+f(-\la)}{2f(\la)}<0 .
$$
By continuity of $H_k(\a,(\z_1,\z_2);c)$ 
with respect to $\a$, $\z_1$ and $\z_2$,
there exist 
$\delta=\delta(c,k)>0$ and $\e=\e(c,k)>0$ for which 
\begin{equation} \label{alpha1}
\sup_{\alpha\in [1-\delta,1]}
H_k(\a,(1-\d,\d);c)
 \leq -\e .
\end{equation}
This establishes the claim for 
$\alpha\in [1-\delta,1]$.

For $\a \in (\frac12,1-\d)$,
let $\zzz=(\z_1,\z_2)$
be given by $\z_1(\a)=\z_2(\ab)$, the latter 
determined by Claims \ref{asmall}--\ref{alarge}, and likewise 
$\z_2(\a)=\z_1(\ab)$.
Then, 
\begin{align*}
H_k(\a,\zzz(\a);c) 
 &= H_k(\a,(\z_2(\ab),\z_1(\ab));c) 
 \\& \leq H_k(\ab,(\z_1(\ab),\z_2(\ab));c) 
 = H_k(\ab,\zzz(\ab);c) .
\end{align*}
The 
inequality follows from \eqref{H=}:
for the first three terms of its right hand side by symmetry, 
and for its last term by applying the inequality $f(x) \geq f(-x)$,
with $x:=\z_2(\ab)-\z_1(\ab) \geq 0$
(in the proofs of Claims \ref{asmall}--\ref{alarge}, $\z_2 \geq \z_1$).
It follows that
\begin{align*}
H_k(\a,\zzz(\a);c) 
 \leq -\e,
\end{align*}
where $\e$ is chosen as the minimum of corresponding values 
in Claims \ref{asmall}--\ref{alarge},
with the value of $\d$ chosen for \eqref{alpha1}
also serving as the $\d$ in Claim~\ref{asmall}.

Finally, for $\zn(c,k)>0$ suitably chosen, we have
$\z_2 \geq \zn(c,k)>0$.
This follows because 
for $\a \geq 1-\d$ we have $\z_2=\d$, while
for $\a \in (1/2,1-\d)$ we have 
$\z_2(\a)=\z_1(\ab)$,
which by Claims \ref{asmall}--\ref{alarge} is
variously of order $\Theta(\ab^{1/2})$ or $\Theta(\ab)$,
and in either case bounded away from 0 since $\ab \geq \d$.
\end{proof}

This completes the claims used in proving Lemma~\ref{Hlemma}.

\section{More precise threshold behavior} \label{sharper}
With relatively little additional work, 
we can prove the prove the finer-grained threshold behavior
given by Theorem~\ref{sharp}.

\begin{proof}[Proof of Theorem~\ref{sharp}]
By a standard and general argument we may assume that $m/n$ has a limit.
We reason contrapositively.
If there is a sequence of $m$ and $n$
for which the desired probability fails to approach 1 as claimed,
then it has a subsequence for which the probability approaches a 
value less than~1,
it in turn has a sub-subsequence for which $\lim m/n$ exists,
and by hypothesis it satisfies $2/k < \lim m/n \leq \infty$.
That is, if there is a counterexample, 
then there is one in which $m/n$ has a limit.
The case $\lim m/n \neq 1$ was already treated by Theorem~\ref{main},
so we assume henceforth that $\lim m/n=1$.

The unsatisfiable part is immediate from Remark~\ref{>1}.
For satisfiability, we begin with $k \geq 4$ and treat $k=3$ at the end.

\smallskip \noindent \textbf{Case} $\mathbf{k \geq 4.} \quad$
Claims \ref{asmall} and \ref{amed}
already treat $m/n=c$ in a closed interval including 1. 
So does Claim~\ref{ahalf}, in its treatment of $\a$ near 1
and the symmetry argument elsewhere,
contingent upon Claim \ref{alarge}.
Claim~\ref{alarge} also allows $c=1$ 
except in the case addressed by \eqref{keycase},
so we need only treat this case.

For $n$ sufficiently large we will have $c \in [0.99,1]$,
thus $\la = \psi^{-1}(c k) > 3.5$
and, for $x \leq x_0 = \twicelo$ as before,
$R(\l,x) \leq R(3.5,x_0) < 0.4$.
Using this to re-treat \eqref{keycase},
\begin{align}
\HH &\leq 
   -c \ln\parens{\tfrac{x^2+2}4}
   + \ln\parens{\tfrac{x^2+2}4}
   - \ln\parens{\tfrac{x^2+2}4}
    + \ln\parens{\tfrac12 + \tfrac12 \cdot 0.4 \, x^2}
 \notag \\& = (1-c) \ln\parens{\tfrac{x^2+2}4}
  + \ln\parens{ \frac{0.8x^2+2}{x^2+2}}
 \notag \\& \leq -0.59 (1-c) - \tfrac1{15}x^2 
  \quad \text{(for $x \leq x_0$)}.
  \label{criticalxxx}
\end{align}

\newcommand{\szzyy}{\sqrt{z_1^2-y^2}\,}
\newcommand{\abscosh}{| {\cosh z_1 e^{i\vartheta} -1} |}

We will also need bounds on $\alnu$ and $\ex[\Ymnl]$
better than those in
\eqref{abless} and \eqref{Ymnlless}.
Reasoning as for \eqref{blessbetter}, from \eqref{aseries} we have
\begin{align}
\alnu
 &=
 \frac{(k\ell)!}{2\pi}\!\!
 \oint\limits_{{\mbox{\Large${z=z_1e^{i\vartheta} \colon 
     \atop\vartheta\in (-\pi,\pi]}$}}\!\!}
 \frac {(\cosh z -1)^{\nu}} {z^{k\ell+1}} \, dz
  \leq
 \frac{(k\ell)!}{2\pi}\!\!
 \int_{-\pi}^\pi
 \frac {\abscosh^{\nu}} {z_1^{k\ell+1}} 
  \, d\vartheta 
 \notag \\&=
 O(1) \: {(k\ell)!}
 \frac {(\cosh z_1 -1)^{\nu}} {z_1^{k\ell+1}} \!\!
 \int_{0}^{\pi/2}
 \parens{ \frac {\abscosh} {\cosh z_1-1}
        }^{\nu} 
  \, d\vartheta .
 \label{a2}
\end{align}
Substituting $z_1 e^{i\vartheta} = x+\im y$, i.e.,
$y=z_1 \sin(\vartheta)$ and
$x=z_1 \cos (\vartheta) =\szzyy$,
\begin{align*}
\abscosh
  &= \cosh\big(\szzyy \big)-\cos(y) , \quad \text{and}
 \\
\frac{d}{dy}
 \abscosh
 &=
 - \, \tfrac{\sinh\parens{\szzyy}}{\szzyy} y + \sin(y)
 \leq
 -[1+\tfrac16 (z_1^2-y^2)] \, y +y,
\end{align*}
using 
$
 \sinh(x)/x 
 \geq 1+\tfrac16 x^2
$,
and
$\sin(y)\leq y$ for $y \geq 0$.
This gives
\begin{align*}
 \abscosh
  - (\cosh z_1 -1)
 &\leq
 -\tfrac16 \int_0^{z_1 \sin \vartheta} (z_1^2-y^2) \, y \, dy
 \\&=
 -\tfrac1{6} (\tfrac12 z_1^2- \tfrac14 y^2) y^2 \Big\lvert_{y=0}^{z_1 \sin \vartheta}
 \leq
 -\tfrac1{24} z_1^4 \sin^2(\vartheta) .
\end{align*}
Immediately,
\begin{align*}
 \ln \frac{\abscosh}{\cosh z_1 -1}
 &\leq
 \ln\parens{ 1 - \frac{\tfrac1{24} z_1^4 \sin^2(\vartheta)}{\cosh z_1-1} }
  \leq
  - \frac{\tfrac1{24} z_1^4 \sin^2(\vartheta)}{\cosh z_1-1} 
 = - \Theta(\vartheta^2) ,
\end{align*}
since $z_1=\Theta(1)$
and $0.63 \vartheta \leq \sin(\vartheta) \leq \vartheta$ 
for $\vartheta \in [0, \pi/2]$.
From \eqref{a2}, then,
\begin{align*}
\alnu
 &=
 O(1) \: {(k\ell)!}
 \frac {(\cosh z_1 -1)^{\nu}} {z_1^{k\ell+1}} \!\!
 \int_{0}^{\pi/2}
  \exp\parens{-\nu \: \Omega(\vartheta^2)}
  \, d\vartheta 
\\ &=
 O(1) \frac1{\sqrt \nu} \:  {(k\ell)!}
 \frac {(\cosh z_1 -1)^{\nu}} {z_1^{k\ell+1}} .
\end{align*}
Comparing with \eqref{abless}, note the leading $1/\sqrt \nu$.
(There is also a new factor $1/z_1$, 
but this is $\Theta(1)$.)

This immediately gives an analog to \eqref{peq}, namely 
\begin{equation*}
p(\ell,\nu)\leb \sqrt{\frac{\la}{\nu z_2}} \,
 \binom{km}{k\ell}^{-1} \frac{\la^{km}}{z_1^{k\ell+1}\,z_2^{k(m-\ell)}} \,\,
\frac{[e^{z_2}(\cosh z_1-1)]^{\nu}f(z_2)^{n-\nu}}{f(\la)^n} .
\end{equation*}
We would like to sum $\binom n \nu$ times this bound
as in \eqref{EYmn1}, but the leading $1/\sqrt \nu$ blocks application of the
binomial theorem.
However, a quick look at the ratio of consecutive terms,
$$
\sqrt\frac{\nu}{\nu+1} \;\; \frac{n-\nu}{\nu+1} \;\;
 \frac{e^{z_2}(\cosh z_1-1)} {f(z_2)} ,
$$
shows that the maximum occurs where the ratio is 1,
at some $\nu_0=\Theta(n)$,
and that terms before $\nu_0/2$ are exponentially smaller than the maximum.
Discarding these terms (which have negligible contribution to the sum)
in the remaining terms
we may replace the $1/\sqrt \nu$ by $1/\sqrt n$
(after which we may add back the early terms with the same substitution).
We then apply the binomial theorem to get an analog of
\eqref{Ymnlless} but smaller by $1/\sqrt n$.
The same simplifications as 
for \eqref{YlesseH} then yield
\begin{equation*}
\ex\bigl[\Ymnl\bigr]
  \leb \,
 {\frac{1}{\sqrt n}}\,
 \exp\bigl[n H_k(\alpha,\zzz;c)\bigr],
\quad \text{for any $\zzz>0$ with $\zzz = \Theta(1)$.}
\end{equation*}

Returning to \eqref{criticalxxx}, 
the contribution of values 
$x \in [0,\twicelo]$
(that is, $\ell \in [\medup n,n/2]$)
to the unsatisfiability probability is
\begin{align*}
\sum 
 \ex[\Ymnl]
 & \leq
\sum \Ortn \exp\parens{ n \HH}
 \\ & \leq
\Ortn \sum \exp\parens{ -0.59 n(1-c) - \tfrac1{15}x^2 n }
  \\ & = 
\Ortn \exp\parens{ 0.59 \: \omegan} \sum \exp\parens{ - \tfrac1{15}x^2 n}
 \\ &= 
\Ortn \exp(-0.59 \: \omegan) \; O(\sqrt n)
 \\ & 
 = \exp(-0.59 \omegan)
 \to 0, 
\end{align*}
since by hypothesis $\omegan \to \infty$.
Adding the contributions
to the unsatisfiability probability from other values of $\ell$,
notably $\ell=2$, the probability that the formula is unsatisfiable is
$O(m^{-(k-2)}+\exp(-0.59 \: \omegan))$.

\medskip

\textbf{Case} \noindent $\mathbf{k=3.} \quad$
We apply similar reasoning.
Since $\a_3 > 0.1$,
values $\a \in (0,0.099]$ are covered by Claim~\ref{asmall}.

For $\a \in [0.099, 0.4]$,
at multiples of $0.001$ we explicitly find values 
$\z_1,\z_2$ (again, multiples of $0.001$)
minimizing $\HH$;
for example $\a=0.300$ yields $\zzz=(0.360,0.667)$.
Using interval arithmetic, we verify that the same $\zzz$ value
yields negative values of $\HH$ for all values of $\a$ between the chosen one 
and the next (here, $\a \in [0.300,0.301]$),
simply by looking at the extreme values of the possible results
in each component calculation for $\HH$ (see \eqref{H=}).  
Since $c$ may be assumed to be arbitrarily close to 1,
allowing for some (sufficiently small) range of $c$ requires no 
further checking.
This proves that $\HH \leq -0.002$ on $\a \in [0.099, 0.400]$,
for all $c$ in a sufficiently narrow range about~1.

For $\a \in [0.4,\tfrac12]$ we follow the same approach as before.
This range equates to $x \in [0, 0.2]$.
Here, $\psi^{-1}(3) > 2.149$ gives
$R(\l,x) \leq R(2.149,0.2) \leq 0.495$,
which is good enough to yield an equivalent of \eqref{criticalxxx}
albeit with a constant smaller than $\tfrac1{15}$.

Values $\a >1/2$ are treated by Claim~\ref{ahalf} just as before.
\end{proof}

\newcommand{\gk}{g_k}
\newcommand{\gkx}{\gk(x)}
\newcommand{\chat}{\hat{c}}
\newcommand{\core}{H_2}
\newcommand{\mus}{{\mu^*}}

\section{\texorpdfstring%
{Satisfiability threshold for unconstrained $k$-XORSAT}%
{Satisfiability threshold for unconstrained k-XORSAT}}
\label{Unconstrained}
If a variable appears in at most one equation,
then deleting that variable, along with the corresponding equation if any,
yields a linear system that, clearly, 
is solvable if and only if the original system was.
Stop this process when each variable appears in at least two equations,
or when the system is empty.
Dubois and Mandler analyzed unconstrained $3$-XORSAT by analyzing this
process, which ends with a (possibly empty) constrained 3-XORSAT instance.

Regarding each variable as a vertex and each equation as a hyperedge
on its $k$ variables yields 
the $k$-uniform ``constraint hypergraph'' underlying a $k$-XORSAT instance.
The process described simply restricts the instance to the 2-core of its
hypergraph.
The analysis by Dubois and Mandler for 3-XORSAT is easily generalized
to $k$-XORSAT using the (later) analyses of the 2-core of
a random $k$-uniform hypergraph, 
and we take this approach.

It is well known that the 2-core of a uniformly random $k$-uniform hypergraph
is, conditioned on its size and order, uniformly random among all
such $k$-uniform hypergraphs with minimum degree~2.
(One short and simple proof is identical to that for
conditioning on the core's degree sequence in \cite[Claim~1]{Molloy}.)
Also, the ``core'' of a random $k$-XORSAT instance
is an instance uniformly random on its underlying hypergraph:
the (uniform) hypergraph core determines the core $A$ matrix,
while the core $b$ is simply the restriction of 
its uniformly random initial value to the surviving rows of $A$,
a process oblivious to $b$.

Thus, satisfiability of a random unconstrained instance hinges on the
edges-to-vertices ratio
of the core of its 
constraint hypergraph.

Recall the definition of $\l$ from \eqref{lambda}.

\begin{theorem} \label{unconstrained}
Let $Ax=b$ be a uniformly random unconstrained uniform random $k$-XORSAT
system with $m$ equations and $n$ variables.
Suppose that $k \geq 3$ and $m/n \to \infty$ with $\lim m/n=c$.
Define 
\begin{align*}
  \gkx := \frac{x}{k (1-e^{-x})^{k-1}} .
\end{align*}
With $\cstark = \gk(\l(k))$,
if $c<\cstark$ then $Ax=b$ is almost surely satisfiable,
and if $c>\cstark$ then $Ax=b$ is almost surely unsatisfiable.
\end{theorem}

\begin{proof}
We treat $k$ as fixed.
Restricting consideration to $x>0$,
from Molloy \cite[proof of Lemma 4]{Molloy},
$\gkx$ has a unique minimum $\chat$,
with $\gkx=c$ having no solutions for any $c<\chat$,
and two solutions for any $c>\chat$.
Simple calculus confirms
that for $k \geq 3$, $\gk$ is unimodal
(indeed, convex).

Let $H$ be a random $k$-uniform hypergraph with $m$ edges and $n$ vertices.
Molloy \cite[Theorem~1]{Molloy}
shows that if $\lim m/n<\chat$ then the 2-core is almost surely empty,
while if $\lim m/n=c>\chat$,
then with $\mu$ the larger solution of $\gk(\mu)=c$,
the order $\nn$ and size $\mm$ of the 2-core almost surely satisfy
\begin{align*}
\nn &= n \, \frac{e^\mu-1-\mu}{e^\mu} + o(n) ,
 & 
\mm &= n \, \frac {\mu (e^\mu-1)}{k e^\mu} + o(n) ;
\end{align*}
see also Achlioptas and Molloy \cite[Proposition 30]{AM}.
It follows for the core that, almost surely,
\begin{align}
\frac \mm \nn
 &= \frac {\mu (e^\mu-1)} {k (e^\mu-1-\mu)} + o(1) 
 = \frac1k \psi(\mu) + o(1) .
 \label{coremn}
\end{align}

Define $\mus=\lambda(k)$ so that $\psi(\mus)=k$;
remember from \eqref{lambda} that for $k>2$ this is well defined,
with $\mus>0$.
We claim that $\mus$ is the larger of the two values of $\mu$ for
which $\gk(\mu) = \gk(\mus)$.
Given that $\gk$ is unimodal, this is true iff $\gk'(\mus)>0$.
Now,
$$
\gk'(\mus) = \frac{1+e^{-\mus} (\mus-\mus k-1)}{k(1+e^{-\mus})^k}
.
$$
Focusing on the numerator, multiplying through by $e^\mus$,
and replacing $k=\psi(\mus)$,
this means showing that 
\begin{align*}
 e^\mus+\mus-1 - \mus
  \parens{ \frac{\mus(e^\mus-1-\mus)}{e^\mus-1} } 
 &>0 .
\end{align*}
Multiplying the expression by $e^\mus-1$ gives
\begin{multline*}
 (e^\mus+\mus-1) (e^\mus-1)
  - \mus^2 (e^\mus-1-\mus)
  \\ =
  (e^\mus-1-\mus-\tfrac12 \mus^2)^2 
   + 3 (e^\mus-1-\tfrac13 \mus-\tfrac1{12} \mus^3) > 0
\end{multline*}
as desired. The inequality is immediate from the 
Taylor series for $e^\mus$, as $\mus>0$.

Let $\cstark = \gk(\mus)$.
Because $\mus$ is the larger of the two values $\mu$ for which
$\gk(\mu)=\gk(\mus)$, we may apply \eqref{coremn},
concluding that a random $k$-uniform hypergraph with 
$\lim m/n = \cstark = \gk(\mus)$
has a core where, almost surely, $\mm/\nn = \frac1k \psi(\mus)+o(1)= 1+o(1)$.

For any $c >\cstark$, the larger solution $\mu$ of $\gk(\mu)=c$
has $\mu>\mus$ (by the unimodality of $\gk$),
and $\psi(\mu)>\psi(\mus)=k$ 
(by Remark~\ref{psi''}).
Thus, a random $k$-uniform hypergraph with 
$\lim m/n = c > \cstark$
has a core where, almost surely, $\mm/\nn = \frac1k \psi(\mu)+o(1) > 1$.
By this section's introductory remarks 
it follows that a random $k$-XORSAT instance with
$\lim m/n = c > \cstark$
reduces to a random constrained $k$-XORSAT instance 
with $\mm/\nn$ converging in probability to a value greater than $1$,
the reduced instance is almost surely unsatisfiable,
and thus so is the original instance.

By the same token,
if $c < \cstark$ then 
either $\gk(\mu)=c$ has no solution (if $c<\chat$), or
its larger solution 
has $\mu<\mus$ 
and $\psi(\mu)<\psi(\mus)=k$.
Thus, a random $k$-XORSAT instance with
$\lim m/n = c < \cstark$
reduces to a constrained $k$-XORSAT instance that 
either is almost surely empty (and trivially satisfied),
or has $\mm/\nn$ converging in probability to a value less then $1$,
and thus is almost surely satisfiable by Theorem~\ref{main}. 
Thus the original instance is almost surely satisfiable.
\end{proof}

\section*{Acknowledgments}
Our sincere thanks go to the anonymous referees for their 
very careful reading and many helpful critical comments.
We are very grateful to Paul Balister for a sketch of a proof that $\HH<0$
(the essence of Lemma~\ref{Hlemma})
using patchwork functional approximation
to get finitely away from the boundaries,
and interval arithmetic for the interior \cite{Balister};
the proof here is a different implementation of those ideas.
We are also grateful to Mike Molloy for helpful comments,
to Colin Cooper and Alan Frieze for pointing out related work,
and to Noga Alon for suggesting we aim for Theorem~\ref{sharp}.


\begin{thebibliography}{99}

\bibitem{AM}
D. Achlioptas and M. Molloy, 
The solution space geometry of random linear equations
\textit{Random Struct. Algorithms}, to appear, 
DOI 10.1002/rsa.20494.


\bibitem{ArFrPi}
J. Aronson, A. Frieze and B. Pittel, On maximum matching in sparse random graphs: Karp-Sipser revisited, 
\textit{Random Struct. Algorithms}, \textbf{12}(2) (1998), 111--177.

\bibitem{Balister}
P. Balister, Personal communication, Apr.\ 2012. 

\bibitem{BlKaWe}
J. Bl\"omer, R. Karp and E. Welzl,
The rank of sparse random matrices over finite fields,
\textit{Random Struct. Algorithms}, \textbf{10}(4) (1997), 407--419.

\bibitem{BBCKW}
B. Bollob\'{a}s, C.Borgs, J. T. Chayes, Jeong-Han Kim, and D. B. Wilson, 
The scaling window of the $2$-SAT transition,
\textit{Random Struct. Algorithms}, \textbf{18}(3) (2001) 201--256.

\bibitem{ChRe}
V. Chv\'{a}tal and B. Reed, Mick gets some (the odds are on his
 side), \textit{33th Annual Symposium on Foundations of Computer Science (Pittsburgh,
 PA, 1992), IEEE Comput. Soc. Press, Los Alamitos, CA} (1992) 620--627.

\bibitem{Cooper2}
C. Cooper,
On the rank of random matrices,
\textit{Random Struct. Algorithms}
\textbf{16}(2) (2000) 209--232.

\bibitem{Cooper1}
C. Cooper,
The cores of random hypergraphs with a given degree sequence,
\textit{Random Struct. Algorithms}
\textbf{25}(4) (2004) 353--375.

\bibitem{CGHS}
D. Coppersmith, D. Gamarnik, M. T. Hajiaghayi and G. B. Sorkin, 
Random MAX SAT, random MAX CUT, and their phase transitions,
\textit{Random Struct. Algorithms}
\textbf{24}(4) (2004) 502--545.

\bibitem{CrDa}
N. Creignon and H. Daud\'e, Smooth and sharp thresholds for random
$k$-XOR-CNF satisfiability, \textit{Theor. Inform. Appl.}\textbf{37} (2003) 127--147.

\bibitem{DPWZ}
R.W.R. Darling, M.D. Penrose, A.R. Wade and S.L. Zabell,
Rank deficiency in sparse random GF[2] matrices,
\textit{arXiv:1211.5455v1} (2012).

\bibitem{DaRa}
H. Daud\'e and V. Ravelomanana, Random $2$-XORSAT at the satisfiability
threshold, \textit{LATIN 2008: Theoretical Informatics, 8th Latin AMerican Symposium
Proceedings} (2008) 12--23.

\bibitem{DGMMPR10}
M. Dietzfelbinger, A. Goerdt, M. Mitzenmacher, A. Montanari, R. Pagh, and M. Rink,
Tight thresholds for cuckoo hashing via XORSAT,
\textit{Proceedings of the 37th International Colloquium on Automata, Languages and Programming} (ICALP'10),
S. Abramsky, C. Gavoille, C. Kirchner, F.M. Auf Der Heide, and P.G. Spirakis (Eds.)
(Springer-Verlag, Berlin, Heidelberg) (2010) 213--225.

\bibitem{DGMMPR09}
M. Dietzfelbinger, A. Goerdt, M. Mitzenmacher, A. Montanari, R. Pagh and M. Rink,
Tight thresholds for cuckoo hashing via XORSAT, 
\textit{arXiv:0912.0287v3} (2010).

\bibitem{DMfocs}
O. Dubois and J. Mandler, The $3$-XORSAT threshold,\textit{
 Proceedings of the 43rd Annual {IEEE} Symposium on Foundations of Computer
 Science, FOCS 2002 (Vancouver, BC, Canada), IEEE Computer Society} (2002)
 769--778.

\bibitem{DM02}
O. Dubois and J. Mandler, The $3$-XORSAT threshold,\textit{
 C. R. Acad. Sci. Paris, Ser. I}
 \textbf{335} (2002) 963--966.

\bibitem{FdlV}
W. Fernandez de~la Vega, On random $2$-SAT, manuscript (1992).

\bibitem{Friedgut}
E. Friedgut, Necessary and sufficient conditions for sharp thresholds
 of graph properties, and the $k$-SAT problem, \textit{J. Amer. Math. Soc.}
 \textbf{12} (1999), 1017--1054.

\bibitem{Goerdt}
A. Goerdt, A threshold for unsatisfiability, \textit{J. Comput. System Sci.}
 \textbf{53} (1996) 469--486.

\bibitem{Kim}
J.H. Kim,
Poisson Cloning Model for Random Graphs
\textit{arXiv:0805.4133} (2008).

\bibitem{Kolchin}
V.F. Kolchin, Random graphs, \textit{Encyclopedia of Mathematics and its
 Applications, vol.~53, Cambridge University Press, Cambridge} (1999).

\bibitem{Molloy}
M. Molloy, Cores in random hypergraphs and Boolean formulas,
 \textit{Random Struct. Algorithms} \textbf{27}(1) (2005) 124--135.

\bibitem{Pittel86}
B. Pittel, Paths in a Random Digital Tree: Limiting Distributions,
\textit{Adv. Appl. Prob.} \textbf{18} (1986) 139--155.

\bibitem{PiSpWo}
B. Pittel, J. Spencer, and N. Wormald, 
Sudden emergence of a giant $k$-core in a random graph,
\textit{J. Combin. Theory Ser B} \textbf{67} (1996), 111--151.

\bibitem{PiYe}
B. Pittel and J.-A Yeum, How frequently is a system of $2$-linear equations
solvable? \textit{Electronic J. Combin.} \textbf{17} (2010) \# R 92.

\end{thebibliography}
\end{document}